\documentclass[leqno,english]{amsart}
\usepackage{amsfonts,amssymb,amsmath,amsgen,amsthm}%,mathabx}

\usepackage{graphicx}
\usepackage{tikz,tkz-tab}
\usepackage{hyperref}
\usepackage{color}
\newcommand{\msc}[2][2000]{%
  \let\@oldtitle\@title%
  \gdef\@title{\@oldtitle\footnotetext{#1 \emph{Mathematics subject
        classification.} #2}}% 
}

\theoremstyle{plain}
\newtheorem{theorem}{Theorem}[section]

\newtheorem{lemma}[theorem]{Lemma}
\newtheorem{corollary}[theorem]{Corollary}
\newtheorem{proposition}[theorem]{Proposition}

\theoremstyle{remark}
\newtheorem{remark}[theorem]{Remark}

\newtheorem{example}[theorem]{Example}

\def\C{{\mathbb C}}% complex numbers
\def\R{{\mathbb R}}% real numbers
\def\N{{\mathbb N}}% nonnegative integers
\def\H{{\mathcal H}}
\def\({\left(}
\def\){\right)}
\def\<{\left\langle}
\def\>{\right\rangle}
\def\le{\leqslant}
\def\ge{\geqslant}

\def\d{{\partial}}
\def\eps{\varepsilon}

\def\si{{\sigma}}

\def\1{{\mathbf 1}}
\def\w{{\tt w}}

\DeclareMathOperator{\RE}{Re}
\DeclareMathOperator{\IM}{Im}
\DeclareMathOperator{\DIV}{div}
\def\Dt{{\Delta t}}
\def\E{{\mathcal E}}
\def\L{\mathcal L}

\numberwithin{equation}{section}

\newcommand{\mynegspace}{\hspace{-0.12em}}
\newcommand{\vvvert}{\rvert\mynegspace\rvert\mynegspace\rvert}

\begin{document}

\title[Splitting for NLS in the semi-classical limit]{On Fourier
  time-splitting methods for nonlinear Schr\"odinger equations
  in the semi-classical limit~II. Analytic regularity} 

\author[R. Carles]{R\'emi Carles}
\author[C. Gallo]{Cl\'ement Gallo}
\address{CNRS \& Univ. Montpellier\\Institut Montpelli\'erain
  Alexander Grothendieck
\\CC51\\Place E. Bataillon\\34095 Montpellier\\ France}
\email{Remi.Carles@math.cnrs.fr}
\email{Clement.Gallo@umontpellier.fr}

\begin{abstract}
  We consider the time discretization based on Lie-Trotter splitting, for the
  nonlinear Schr\"odinger equation, in the semi-classical limit, with
  initial data under the form of WKB states. We show that both the
  exact and the numerical solutions keep a WKB structure, on a time
  interval independent of the Planck constant. We prove error
  estimates, which show that the quadratic observables can be computed
  with a time step independent of the Planck constant. The functional
  framework is based on time-dependent analytic spaces, in order to
  overcome a previously encountered loss of regularity phenomenon. 
\end{abstract}
\thanks{This work was supported  by the French ANR projects
  SchEq (ANR-12-JS01-0005-01) and BECASIM
  (ANR-12-MONU-0007-04)}  
\maketitle

\section{Introduction}
\label{sec:intro}

This paper is devoted to the analysis of the numerical approximation
of the solution to 
\begin{equation}
  \label{eq:nls}
  i\eps\d_t u^\eps +\frac{\eps^2}{2}\Delta u^\eps = \lambda
  |u^\eps|^{2\si} u^\eps,\quad (t,x)\in [0,T]\times \R^d,
\end{equation}
in the semi-classical limit $\eps \to 0$.  The nonlinearity is smooth
and real-valued: $\lambda\in \R$ and $\si\in \N$. The initial data that we
consider are BKW states:
\begin{equation}
  \label{eq:ci-u}
  u^\eps(0,x)=a_0(x)e^{i\phi_0(x)/\eps}=:u_0^\eps(x),
\end{equation}
where $\phi_0:\R^d\to \R$ is a real-valued phase, and $a_0:\R^d\to \C$
is a possibly complex-valued amplitude. An important feature of such
initial data is that in the context of the semi-classical limit for
\eqref{eq:nls}, they yield solution which are in $L^\infty(\R^d)$
uniformly in $\eps$, at least on some time interval $[0,T]$ for some
$T>0$ independent of $\eps$. Also, not that even if $\phi_0=0$ (no
rapid oscillation initially), for $\tau>0$ arbitrarily small and independent of $\eps$, $u^\eps(\tau)$ takes the form of a WKB state as in
\eqref{eq:ci-u} with amplitude and phase solving \eqref{eq:syst-grenier}--\eqref{eq:ci} below (see \cite{CaBook}). Note that even if $\phi_0=0$, the coupling shows that $\phi^\eps$ becomes non-trivial
instantaneously.
\smallbreak

We consider more precisely the time discretization for \eqref{eq:nls}
based on Fourier time splitting.  We denote by $X^t_\eps$ the map
$v^\eps(0,\cdot)\mapsto v^\eps(t,\cdot)$, where
\begin{equation}\label{eq:linear}
  i\eps\d_t v^\eps +\frac{\eps^2}{2}\Delta v^\eps=0. 
\end{equation}
We also denote by $Y^t_\eps$ the map
$w^\eps(0,\cdot)\mapsto w^\eps(t,\cdot)$, where
\begin{equation}\label{eq:Y}
  i\eps\d_t w^\eps = \lambda |w^\eps|^{2\si}w^\eps.
\end{equation}
We consider the Lie-Trotter type splitting operator
\begin{equation}\label{eq:Z}
  Z_{\eps}^t =Y^t_\eps X^t_\eps .
\end{equation}
The Lie-Trotter operator $X^t_\eps Y^t_\eps $ could be handled in the
same fashion. The advantage of splitting methods is that they involve
sub-equations which are simpler to solve than the initial equation. In
our case, \eqref{eq:linear} is solved explicitly by using the Fourier
transform, defined by 
 \begin{equation*}
    \widehat \psi(\xi) = \frac{1}{(2\pi)^{d/2}}\int_{\R^d} e^{-ix\cdot
      \xi}\psi(x)dx, 
  \end{equation*}
since it becomes an ordinary differential equation
\begin{equation}\label{eq:X}
  i\eps \d_t \widehat v^\eps -\frac{\eps^2}{2}|\xi|^2 \widehat v^\eps =0,
\end{equation}
hence
\begin{equation*}
  \widehat {X^t_\eps v}(\xi)=e^{-i\eps \frac{t}{2}|\xi|^2}\widehat v(\xi). 
\end{equation*}
Also, since $\lambda\in \R$, in \eqref{eq:Y} the modulus of $w^\eps$
does not depend on time, hence
\begin{equation}\label{eq:Yexpl}
  Y^t_\eps w(x) = w(x) e^{-i\lambda \frac{t}{\eps} |w(x)|^{2\si}}.
\end{equation}
In the case $\eps=1$, several results exist to prove that the
Lie-Trotter time splitting is of order one, and the Strang splitting 
of order two (\cite{BBD,Lu08}). The drawback of these proofs is
that they rely on uniform Sobolev bounds for the exact solution, of
the form $u\in L^\infty([0,T];H^s(\R^d))$, for $s\ge 2$. However, in
the framework of \eqref{eq:nls}, these norms are not uniformly
bounded as $\eps\to 0$, in the sense that we rather have $\|u^\eps(t)\|_{H^s}\approx\eps^{-s}$, due to the oscillatory nature of $u^\eps$. 
\smallbreak

In the case of a linear potential ($|u^\eps|^{2\si}$ is replaced by a
known function of $x$ in \eqref{eq:nls}), error estimates are given in
\cite{BJM1}; see also \cite{DeTh10,DeTh13}. In the nonlinear case,
error estimates are established in \cite{CaSINUM}, but for other nonlinearities than in \eqref{eq:nls}--\eqref{eq:ci-u}. The proof there requires either to consider a weakly nonlinear regime, that is \eqref{eq:nls} is replaced by 
\begin{equation*}
 i\eps\d_t u^\eps +\frac{\eps^2}{2}\Delta u^\eps = \eps\lambda
  |u^\eps|^{2\si} u^\eps,\quad (t,x)\in [0,T]\times \R^d,
\end{equation*}
with the same initial data \eqref{eq:ci-u}, or to replace the
nonlinearity in \eqref{eq:nls} with a smoothing nonlinearity of Poisson
type. We recall in Section~\ref{sec:overview} why these assumptions
are made in \cite{CaSINUM}. The goal of this paper is to prove error
estimates which are similar to those established in \cite{CaSINUM}, but
for \eqref{eq:nls}--\eqref{eq:ci-u}. Before stating our main result,
we introduce a few notations. The Fourier transform is normalized as
\begin{equation*}
  \hat f(\xi) = \frac{1}{(2\pi)^{d/2}}\int_{\R^d} e^{-ix\cdot \xi}f(x)dx.
\end{equation*}
A tempered distribution $f$ is in $H^s(\R^d)$ if $\xi\mapsto \<\xi\>^s\hat f(\xi)$
  belongs to $L^2(\R^d)$, where 
  \begin{equation*}
    \<\xi\>=\sqrt{1+|\xi|^2}. 
  \end{equation*}
\begin{theorem}\label{theo:main}
 Suppose that $d,\si\in \N$, $d,\si\ge 1$,  and $\lambda\in \R$. Let
  $\phi_0,a_0$ such that
  \begin{equation*}
    \int_{\R^d} e^{\<\xi\>^{1+\delta}}\(|\hat \phi_0(\xi)|^2 + |\hat
      a_0(\xi)|^2\)d\xi<\infty, 
  \end{equation*}
for some $\delta>0$, and $u_0^\eps$ given by \eqref{eq:ci-u}. There
exist $T,\eps_0,c_0>0$ and $(C_k)_{k\in\N}$ such that for all $\eps\in
(0,\eps_0]$, the following holds:\\ 
$1.$ \eqref{eq:nls}-\eqref{eq:ci-u} has a unique solution $u^\eps
=S^t_\eps u_0^\eps\in C([0,T],H^\infty)$, where $H^\infty=\cap_{s\in
  \R}H^s$. Moreover, there exist 
$\phi^\eps$ and $a^\eps$  with, for all $k\in \N$,  
\begin{equation*}
  \sup_{t\in [0,T]}\(\|a^\eps(t)\|_{H^k(\R^d)}+\|
  \phi^\eps(t)\|_{H^k(\R^d)}\)\le 
  C_k,%\quad \forall 
  %\eps\in (0,\eps_0],
\end{equation*}
such that $u^\eps(t,x)=a^\eps(t,x)e^{i\phi^\eps(t,x)/\eps}$ for all
$(t,x)\in [0,T]\times \R^d$.\\
$2.$ For all $\Dt\in (0,c_0]$, for all $n\in \N$ such that $t_n=n\Dt\in
  [0,T]$, there exist $\phi_n^\eps$ and $a_n^\eps$ with, for all $k\in \N$,
\begin{equation*}
  \|a_n^\eps\|_{H^k(\R^d)}+\|\phi^\eps_n\|_{H^{k}(\R^d)}
\le C_k,%\quad \forall
  %\eps\in (0,\eps_0],
\end{equation*}
such that $(Z_\eps^\Dt)^n \(a_0 e^{i\phi_0/\eps}\) = a_n^\eps
e^{i\phi_n/\eps}$.\\
$3.$ For all $\Dt\in (0,c_0]$, for all $n\in \N$ such that $n\Dt\in
  [0,T]$, the following  error estimate holds:
  \begin{equation*}
 \left\| a_n^\eps - a^\eps(t_n)\right\|_{H^{k}}+ \left\| \phi_n^\eps -
   \phi^\eps(t_n)\right\|_{H^{k}}\le C_k \Dt .
 \end{equation*}
\end{theorem}

\begin{example}
  The assumptions of Theorem~\ref{theo:main} are satisfied as soon as
  $\hat\phi_0$ and $\hat a_0$ are compactly supported, or in the case
  of Gaussian functions, typically. 
\end{example}
Note that the first two points of the theorem imply that the functions
$a$ and $\phi$ are not rapidly oscillatory: the oscillatory nature of
both the exact and the numerical solutions is encoded in the
exponential which relates the functions $a$ and $\phi$ to $u$. 

We readily infer error estimates for the wave function and for
quadratic observables,
\begin{align*}
  \text{Position density: }& \rho^\eps(t,x)=|u^\eps(t,x)|^2.\\
\text{Current density: }& J^\eps(t,x) = \eps \IM\(\overline
u^\eps(t,x)\nabla u^\eps(t,x)\).
\end{align*} 
\begin{corollary}\label{cor:wave}
  Under the assumptions of Theorem~\ref{theo:main}, there exist $T>0$, 
  $\eps_0>0$ and 
$C,c_0$ independent of $\eps\in (0,\eps_0]$ such that for
  all $\Dt\in (0,c_0]$, for all $n\in \N$ such that $n\Dt\in [0,T]$, and for all $\eps\in(0,\eps_0]$,
 \begin{align*}
 & \left\| (Z_\eps^\Dt)^n u_0^\eps
   -S^{t_n}_\eps
      u_0^\eps\right\|_{L^2(\R^d)}\le C \frac{\Dt}{\eps},\\
   & \left\| \left\lvert (Z_\eps^\Dt)^n u_0^\eps\right\rvert^2
     -|\rho^\eps(t_n)|^2   \right\|_{L^1(\R^d)\cap L^\infty(\R^d)}\le C \Dt,\\ 
&\left\| \IM\(\eps \overline{(Z_\eps^\Dt)^n u_0^\eps}\nabla
  (Z_\eps^\Dt)^n u_0^\eps\) - J^\eps(t_n) \right\|_{L^1(\R^d)\cap
  L^\infty(\R^d)}\le C \Dt.  
 \end{align*}
\end{corollary}
This result is in agreement with the numerical experiments presented
in \cite{BJM2}: to simulate the wave function $u^\eps$, the time step
must satisfy $\Dt =o(\eps)$, while to observe the quadratic
observables, $\Dt=o(1)$ can be chosen independent of $\eps$.

\section{Overview of the proof}
\label{sec:overview}

We present the general strategy for the proof of
Theorem~\ref{theo:main} in the case of a more general nonlinearity,
\begin{equation}
  \label{eq:nls-gen}
  i\eps\d_t u^\eps +\frac{\eps^2}{2}\Delta u^\eps =f\(
  |u^\eps|^{2}\) u^\eps,
\end{equation}
with $f$ real-valued. For initial data of the form \eqref{eq:ci-u}, it
has been noticed in \cite{CaSINUM} that the numerical discretization
preserves such a structure, in the sense that the numerical solution
satisfies the point 2. in Theorem~\ref{theo:main}. Indeed, the exact
solution can be represented as $u^\eps= a^\eps e^{i\phi^\eps/\eps}$,
where $a^\eps$ and $\phi^\eps$ are given by
\begin{equation}\label{eq:syst-grenier}
  \left\{
    \begin{aligned}
      &\d_t \phi^\eps + \frac{1}{2}|\nabla \phi^\eps|^2 +f\( |a^\eps|^2\)=0 
      ,\\
& \d_t a^\eps + \nabla \phi^\eps \cdot \nabla a^\eps
+\frac{1}{2}a^\eps\Delta \phi^\eps=
\frac{i\eps}{2}\Delta a^\eps,
    \end{aligned}
\right.
\end{equation}
with initial data
\begin{equation}
  \label{eq:ci}
  \phi^\eps_{\mid t=0}=\phi_0,\quad a^\eps_{\mid t=0}=a_0. 
\end{equation}
The main feature of this representation is that even though they must
be expected to depend on $\eps$,  $a^\eps$ and $\phi^\eps$ are
 bounded in Sobolev spaces uniformly in $\eps\in (0,1]$ on some time
 interval $[0,T]$ for some $T$ independent of $\eps$. 
\smallbreak

The idea of representing the solution $u^\eps$ under this form goes
back to E.~Grenier \cite{Grenier98}. The main features of
\eqref{eq:syst-grenier} is that the left hand side defines a
symmetrizable hyperbolic system under the assumption $f'>0$, and the
right hand side is skew adjoint (hence plays no role at the level of
energy estimates). Note that in the case of \eqref{eq:nls}, this
forces $\lambda>0$ and $\si=1$ (cubic defocusing nonlinearity). For a
nonlocal nonlinearity, $f(|u|^2)=K\ast |u|^2$, the approach of Grenier
can easily be adapted if $\hat K$ decays at least like $|\xi|^{-2}$
for large $|\xi|$ (see e.g. \cite{CaSINUM}). The approach of Grenier
has also been generalized to more general nonlinearities: see
\cite{ACARMA,ChRo09} 
for the defocusing case, and \cite{Tho08} for the focusing case.
However, we do not use  these results, as we now discuss.
\smallbreak

Indeed, the splitting scheme for \eqref{eq:nls-gen} amounts to some
splitting technique on \eqref{eq:syst-grenier}. Suppose that one
solves the linear equation \eqref{eq:linear} with initial data
$v^\eps(0)=a_0 e^{i\phi_0/\eps}$. Then the solution $v^\eps$ can be
written as $v^\eps(t)= a^\eps(t) e^{i\phi(t)/\eps}$, with $a^\eps$
and $\phi$ given by
\begin{equation}
  \label{eq:linear-syst}
  \left\{
    \begin{aligned}
      & \d_t \phi+\frac{1}{2}|\nabla \phi|^2=0,\quad \phi_{\mid
        t=0}=\phi_0,\\
&\d_t a^\eps + \nabla \phi\cdot \nabla a^\eps+\frac{1}{2}a^\eps\Delta
\phi=i\frac{\eps}{2}\Delta a^\eps,\quad a^\eps_{\mid t=0} =a_0. 
    \end{aligned}
\right.
\end{equation}
Similarly, the solution to \eqref{eq:Y} with initial data
$w^\eps(0)=a_0 e^{i\phi_0/\eps}$ can be
written as $w^\eps(t)= a(t) e^{i\phi(t)/\eps}$, with $a$
and $\phi$ given by
\begin{equation}
  \label{eq:Y-syst}
  \left\{
    \begin{aligned}
      & \d_t \phi+f(|a|^2)=0,\quad \phi_{\mid
        t=0}=\phi_0,\\
&\d_t a=0,\quad a_{\mid t=0} =a_0. 
    \end{aligned}
\right.
\end{equation}
So computing the numerical solution amounts to solving successively
\eqref{eq:linear-syst} and \eqref{eq:Y-syst}, which turns out to be a
splitting scheme on \eqref{eq:syst-grenier}. We denote by $\mathcal{X}^t_\eps:(\phi_0,a_0)\mapsto (\phi(t,\cdot),a^\eps(t,\cdot))$ the flow for \eqref{eq:linear-syst} and by $\mathcal{Y}^t_\eps:(\phi_0,a_0)\mapsto (\phi(t,\cdot),a(t,\cdot))$ the flow for \eqref{eq:Y-syst}. The Lie-Trotter splitting operator we consider for \eqref{eq:syst-grenier} is then
\begin{equation}
  \label{eq:Zcal}
  \mathcal{Z}^t_\eps=\mathcal{Y}^t_\eps\mathcal{X}^t_\eps
\end{equation}
\smallbreak

Now in the case of a cubic defocusing nonlinearity (which enters the
framework of \cite{Grenier98}), we face a loss of regularity
issue. Indeed, the reason why \eqref{eq:syst-grenier} is convenient
lies first in the structure of the left hand side, which enjoys
symmetry properties: the splitting leading to
\eqref{eq:linear-syst}--\eqref{eq:Y-syst} ruins this property. Suppose
for instance that at time $t=0$, $\phi_0\in H^s(\R^d)$ and $a_0\in
H^k(\R^d)$, for large $s$ and $k$. In \eqref{eq:linear-syst}, the
first equation propagates the $H^s$ regularity on a small time
interval, provided $s$ is large. The second equation shows that
$a^\eps$ cannot be more regular than $H^{s-2}$, due to the last term
of the left hand side. Now if we start with $\phi_0\in H^s$ and
$a_0\in H^{s-2}$ in \eqref{eq:Y-syst} (with $f(|a|^2)=|a|^2$ for a
cubic defocusing nonlinearity), we see that $\phi\in H^{s-2}$ (provided
$s-2>d/2$), and that no better regularity must be expected. So after
one iteration of the operator $Z_\eps^t$, $\phi$ has lost two levels
of regularity. When iterating $Z$ with a small time step $\Dt$, this
loss becomes fatal. This is why in \cite{CaSINUM}, it is assumed that
either $f$ is smoothing (to regain at least two levels of regularity)
or that a factor $\eps$ is present in front of $f$, so that
\eqref{eq:Y-syst} is altered to 
\begin{equation*}
 \left\{
    \begin{aligned}
      & \d_t \phi=0,\quad \phi_{\mid
        t=0}=\phi_0,\\
&\d_t a=-if(|a|^2)a,\quad a_{\mid t=0} =a_0. 
    \end{aligned}
\right.
\end{equation*}
The main technical originality of this paper is based on the remark
that if instead of working in Sobolev spaces, one works in \emph{time
  dependent analytic spaces}, it is possible to control the loss of
regularity. Such an idea goes back to \cite{PGX93}, to solve 
\eqref{eq:syst-grenier}. The fact that we consider decreasing
time dependent weight to measure the analytic regularity is strongly
inspired by the analysis of J.~Ginibre and G.~Velo in the context of
long range scattering for Hartree equations \cite{GV01}, and is also
reminiscent of the functional framework used by J.Y.~Chemin for the Navier-Stokes equation \cite{Ch04} and developed by C.~Mouhot and
C.~Villani to prove Landau damping \cite{MoVi11}. 
\smallbreak

The main technical tools needed here are presented in
Section~\ref{sec:technical}. Thanks to these tools, we can prove that
both the theoretical and the numerical solutions remains analytic in a
suitable sense on some time interval $[0,T]$ with $T>0$ independent of
$\eps$ (Sections~\ref{sec:fundamental} and \ref{sec:bounds}). 
\smallbreak

The next key estimate is the local error estimate, presented in
Section~\ref{sec:local-error}. It is based on the general formula
established in \cite{DeTh13}. As noticed in \cite{CaSINUM}, we must
apply this formula to the system
\eqref{eq:linear-syst}--\eqref{eq:Y-syst} and not only to
\eqref{eq:linear}--\eqref{eq:Y}. 
\smallbreak

With these propagating estimates and the local error estimate, the
proof of Theorem~\ref{theo:main} follows from the trick known as Lady
Windermere's fan. Note however that because of the nonlinear context,
where global bounds for the numerical solutions are not known a
priori,  the argument requires some extra care. We rely on the
induction technique introduced in \cite{HLR13}, which is sufficiently
robust to be readily adapted to our case, as in \cite{CaSINUM}.

\section{Technical background}
\label{sec:technical}

We recall here some of the technical tools introduced in
\cite{GV01}. We state the main properties established there concerning
time dependent Gevrey spaces, and simplify as much as possible the
framework, in view of the present context. 
\smallbreak
 
For $ 0<\nu\le 1$ and $\rho>0$, we introduce the exponential weight 
\begin{equation*}
  \w(\xi)= \exp \(\rho
  \max(1,|\xi|)^\nu\),
\end{equation*}
which is equivalent to $\exp(\rho\<\xi\>^\nu)$. 
Define $u_>$ and $u_<$ by:
\begin{equation*}
  \hat u_<(\xi) = \hat u(\xi)\1_{|\xi|\le 1},\quad  \hat u_>(\xi) =
  \hat u(\xi)\1_{|\xi|> 1}.
\end{equation*}
For $k,\ell\in \R$ and $0\le \ell_<<d/2$, the following families of norms are
defined in \cite{GV01}:
\begin{align*}
 & a\mapsto \(\| |\xi|^k \w(\xi)\hat a_> (\xi)\|_{L^2}^2 + \| 
 \w(\xi)\hat a_< (\xi)\|_{L^2}^2\)^{1/2},\\ 
& \phi\mapsto\( \| |\xi|^{\ell +2} \w(\xi)\hat \phi_> (\xi)\|_{L^2}^2 + \| |\xi|^{\ell_<}
 \w(\xi)\hat \phi_< (\xi)\|_{L^2}^2\)^{1/2}.
\end{align*}
The first norm is well suited to estimate amplitudes, and the second
is adapted to phases. As suggested by the above notations, the indices
will be different for amplitudes and phases. This can be related to
the fact that in the hydrodynamical setting with $\lambda>0$,
\eqref{eq:syst-grenier} with $\eps=0$ is a hyperbolic system in the
unknown $(\nabla \phi,a)$, and not in $(\phi,a)$. Indeed, eventually there will be a
shift of one index between the norm in $\phi$ and the norm in $a$ (see
Lemma~\ref{lem:evol-norm} and Proposition~\ref{prop:local} below). 

 In the properties related to these norms which
will be used in this paper, the value of $\ell_<$ is
irrelevant. Therefore, we set $\ell_<=0$, and consider only one family
of norms: for $\ell\ge 0$, we set
\begin{align*}
&\H_\rho^\ell=\{\psi\in L^2(\R^d),\quad \|\psi\|_{\H_\rho^\ell}<\infty\},\\
\text{where }&
  \|\psi\|_{\H_\rho^\ell}^2 := \| |\xi|^\ell \w(\xi)\hat \psi_> (\xi)\|_{L^2}^2 + \| 
 \w(\xi)\hat \psi_< (\xi)\|_{L^2}^2\sim
               \int_{\R^d}\<\xi\>^{2\ell}e^{2\rho\<\xi\>^\nu}
               |\hat\psi(\xi)|^2d\xi. 
\end{align*}
\begin{remark}
 The above definition is slightly different from the standard
 definition for Gevrey spaces, since low frequencies are smoothed out
 in the definition of the weight $\w$: $\max(1,|\xi|)$ (or $\<\xi\>$) in
 $\w$ is usually replaced with $|\xi|$. 
\end{remark}
Note that the following estimate is a straightforward consequence of
this definition: for any $\alpha\in \N^d$, $\ell\ge 0$,
\begin{equation}
  \label{eq:der}
  \|\d^\alpha \psi\|_{\H^\ell_\rho}\le   \|
  \psi\|_{\H^{\ell+|\alpha|}_\rho},\quad \forall \psi\in \H^{\ell+|\alpha|}_\rho.
\end{equation}
Also, in view of the standard Sobolev embedding, 
\begin{equation*}
  \|\psi\|_{L^\infty(\R^d)} \le C \|\psi\|_{H^s(\R^d)},
\end{equation*}
valid for $s>d/2$, we have
\begin{equation}\label{eq:Linfty}
  \|\psi\|_{L^\infty(\R^d)} \le C \|\psi\|_{\H_\rho^s},
\end{equation}
with the same constant $C$ independent of $\rho\ge 0$. 
\smallbreak

The above notation may seem rather heavy: it is chosen so because the
weight $\rho$ will depend on time, as we now discuss. 
For a time-dependent $\rho$, we have:
\begin{equation}\label{eq:evol-norm}
  \frac{d}{dt} \|\psi\|_{\H_\rho^\ell}^2 = 2\dot \rho
  \|\psi\|_{\H_\rho^{\ell+\nu/2}}^2+2\RE\<\psi,\d_t \psi\>_{\H_\rho^\ell}.
\end{equation}
Even though $\rho$ depends on time, we will consider below
``continuous'' $\H_\rho^\ell$ valued functions. We mean functions that
belong to 
$$C(I, \H_\rho^\ell):=\left\{\psi\in C(I,L^2)\text{ such that
  }\w\psi\in C(I,\H_0^\ell)=C(I,H^\ell)\right\}$$ 
for some interval $I$.

To fix the technical framework once and for all, we recall another
important result from \cite{GV01}. 
Consider the system
\begin{equation}\label{eq:GV0}
  \left\{
    \begin{aligned}
      &\d_t \phi + \frac{1}{2}|
      \nabla\phi|^2 +\lambda \RE \(|\nabla|^{\mu-d}a \overline a\) =0 ,\\
& \d_t a + \nabla \phi \cdot \nabla a
+\frac{1}{2}a\Delta \phi=0,
    \end{aligned}
\right.
\end{equation}
 for
some time interval $I$, and $0<\mu\le d$. 
Lemma~3.5 from \cite{GV01}, which uses \eqref{eq:evol-norm} as well as
rather involved estimates, implies
that  under the assumptions 
\begin{align*}
&  \ell>d/2+1-\nu,\quad k\ge \nu/2,\quad \ell\ge k+1-\nu,\\
&k \ge \ell +\mu-d+1 -\nu,\quad 2k>\ell+\mu-d+1-\nu+d/2,
\end{align*}
any  solution of \eqref{eq:GV0} on  $I$, such
that $(\phi,a)\in  C(I,\H_\rho^{\ell+1}\times \H_\rho^{k})\cap
L^2_{\rm loc}(I, \H_\rho^{\ell+1+\nu/2}\times \H_\rho^{k+\nu/2})$, satisfies
\begin{align*}
 & \left| \d_t \|\phi\|_{\H_\rho^{\ell+1}}^2 - 2\dot \rho
   \|\phi\|_{\H_\rho^{\ell+1+\nu/2}}^2\right|\le 
  C\(\|\phi\|_{\H_\rho^{\ell+1+\nu/2}}^2\|\phi\|_{\H_\rho^{\ell+1}} +
  \|a\|_{\H_\rho^{k+\nu/2}}\|\phi\|_{\H_\rho^{\ell+1+\nu/2}}
\|a\|_{\H_\rho^{k}}\),\\ 
 & \left| \d_t \|a\|_{\H_\rho^{k}}^2 - 2\dot \rho
   \|a\|_{\H_\rho^{k+\nu/2}}^2\right|\le 
  C\(  \|a\|_{\H_\rho^{k+\nu/2}}^2\|\phi\|_{\H_\rho^{\ell+1}} +
  \|a\|_{\H_\rho^{k+\nu/2}}\|\phi\|_{\H_\rho^{\ell+1+\nu/2}}\|a\|_{\H_\rho^{k}}\).
 \end{align*}
In the case of a cubic nonlinearity, we want to set
$\mu=d$. Therefore, the above algebraic conditions 
\begin{equation*}
  \ell \ge k+1-\nu\quad\text{and}\quad k\ge \ell +1-\nu
\end{equation*}
 imply $\nu\ge 1$, hence $\nu=1$ and $k=\ell$. In view of this remark,
 we suppose from now on $\nu=1$, that is,  
we consider analytic functions (see \cite{GV01}).
\smallbreak

Since we consider only analytic functions, we borrow from \cite{GV01}
the only inequalities that we will really use, which appear in
\cite[Lemma~3.4]{GV01}: 
\begin{lemma}\label{lem:tame} Let $m\ge 0$. Then,\\
$1.$ For  $k+s >m+d/2+2$, and $k,s\ge m+1$,
\begin{equation*}
  \| \nabla\phi\cdot \nabla a\|_{\H_\rho^m}\le C
    \|\phi\|_{\H_\rho^{s}} \|a\|_{\H_\rho^k}. 
\end{equation*} 
$2.$ For $k+s >m+2+d/2$, $k\ge m$ and  $s\ge m+2$, 
  \begin{equation*}
  \|  a\Delta \phi\|_{\H_\rho^m}\le C \|\phi\|_{\H_\rho^{s}} \|a\|_{\H_\rho^k}.
\end{equation*}
$3.$ For $s>d/2$,
\begin{equation}\label{eq:tame}
    \|\psi_1 \psi_2\|_{\H_\rho^m}\le C\(
    \|\psi_1\|_{\H_\rho^m}\|\psi_2\|_{\H_\rho^s} +
    \|\psi_1\|_{\H_\rho^s}\|\psi_2\|_{\H_\rho^m}\).  
  \end{equation}
The various constants $C$ are independent of $\rho$.
\end{lemma}
We infer the important lemma:
\begin{lemma}\label{lem:evol-norm}
    Set $\nu=1$, and let  $\si\in \N$, $\lambda\in \R$, $\ell>d/2$, and $I$ be some
    time interval. Let $(\varphi,b)\in   C(I,\H_\rho^{\ell+1}\times \H_\rho^{\ell})\cap
L^2_{\rm loc}(I, \H_\rho^{\ell+3/2}\times \H_\rho^{\ell+1/2})$. Then
any solution
$(\phi,a^\eps)\in   C(I,\H_\rho^{\ell+1}\times \H_\rho^{\ell})\cap
L^2_{\rm loc}(I, \H_\rho^{\ell+3/2}\times \H_\rho^{\ell+1/2})$ to 
\begin{equation}\label{eq:iteration0}
  \left\{
    \begin{aligned}
      &\d_t \phi + \frac{1}{2}\nabla \varphi\cdot
      \nabla\phi +\lambda |b|^{2\si} =0 ,\\
& \d_t a^\eps + \nabla \varphi \cdot \nabla a^\eps
+\frac{1}{2}a^\eps\Delta \varphi=
\frac{i\eps}{2}\Delta a^\eps,
    \end{aligned}
\right.
\end{equation}
satisfies
\begin{align*}
  \left| \d_t \|\phi\|_{\H_\rho^{\ell+1}}^2 - 2\dot \rho
   \|\phi\|_{\H_\rho^{\ell+3/2}}^2\right| &\le 
  C\Big(\|\phi\|_{\H_\rho^{\ell+3/2}}^2\|\varphi\|_{\H_\rho^{\ell+1}}
   +\|\phi\|_{\H_\rho^{\ell+3/2}}\|\varphi\|_{\H_\rho^{\ell+3/2}} \|\phi\|_{\H_\rho^{\ell+1}}\\
&\qquad +
  \|b\|_{\H_\rho^{\ell+1/2}}\|\phi\|_{\H_\rho^{\ell+3/2}}\|b\|_{\H_\rho^{\ell}}^{2\si-1}\Big),\\ 
  \left| \d_t \|a^\eps\|_{\H_\rho^{\ell}}^2 - 2\dot \rho
   \|a^\eps\|_{\H_\rho^{\ell+1/2}}^2\right| & \le 
  C\Big(  \|a^\eps\|_{\H_\rho^{\ell+1/2}}^2\|\varphi\|_{\H_\rho^{\ell+1}} +
  \|a^\eps\|_{\H_\rho^{\ell+1/2}}\|\varphi\|_{\H_\rho^{\ell+3/2}}\|a^\eps\|_{\H_\rho^{\ell}}\Big),
 \end{align*}
where $C$ is independent of $\eps$ and $\rho$. 
 \end{lemma}
 \begin{proof}
   In view of \eqref{eq:evol-norm} and \eqref{eq:iteration0}, we have
   \begin{equation*}
     \d_t \|\phi\|_{\H_\rho^{\ell+1}}^2 - 2\dot \rho
   \|\phi\|_{\H_\rho^{\ell+3/2}}^2 =- \RE\<\phi, \nabla\varphi
   \cdot \nabla \phi\>_{\H_\rho^{\ell+1}} -2\lambda
   \RE\<\phi,|b|^{2\si}\>_{\H_\rho^{\ell+1}}. 
   \end{equation*}
Cauchy-Schwarz inequality yields
\begin{equation*}
   \left\lvert\<\phi, \nabla\varphi
   \cdot \nabla \phi\>_{\H_\rho^{\ell+1}} \right\rvert \le
 \|\phi\|_{\H_\rho^{\ell+3/2}}\| \nabla\varphi
   \cdot \nabla \phi\|_{\H_\rho^{\ell+1/2}}.
\end{equation*}
Inequality \eqref{eq:tame} with $m=\ell+1/2$ and $s=\ell$ yields
\begin{equation}\label{eq:produit-grad}
\begin{aligned}
  \| \nabla\varphi
   \cdot \nabla \phi\|_{\H_\rho^{\ell+1/2}} &\le C \(\|\nabla
   \varphi\|_{\H_\rho^{\ell+1/2}} \|\nabla \phi\|_{\H_\rho^{\ell}}  +\|\nabla
   \varphi\|_{\H_\rho^{\ell}}  \|\nabla   \phi\|_{\H_\rho^{\ell+1/2}} \)\\
&\le C \(\|\varphi\|_{\H_\rho^{\ell+3/2}} \| \phi\|_{\H_\rho^{\ell+1}}  +\|
   \varphi\|_{\H_\rho^{\ell+1}}  \|   \phi\|_{\H_\rho^{\ell+3/2}}  \),
\end{aligned}
\end{equation}
where we have used \eqref{eq:der}. The term involving $b$ can be treated
similarly. Indeed, using \eqref{eq:tame} on the one hand with
$m=\ell+1/2$ and $s=\ell$ and on the other hand with $m=\ell=s$, we
can prove by induction on $\sigma$ that 
\begin{equation}
\||b|^{2\sigma}\|_{\H_\rho^{\ell+1/2}}\le C\| b\|^{2\sigma-1}_{\H_\rho^{\ell}}\|b\|_{\H_\rho^{\ell+1/2}},
\end{equation}
 hence the first inequality for Lemma~\ref{lem:evol-norm}.

For the second inequality,   
\begin{align*}
     \d_t \|a^\eps\|_{\H_\rho^{\ell}}^2 - 2\dot \rho
   \|a^\eps\|_{\H_\rho^{\ell+1/2}}^2 &=- 2\RE\<a^\eps, \nabla\varphi
   \cdot \nabla a^\eps\>_{\H_\rho^{\ell}} -\RE\<a^\eps, a^\eps\Delta\varphi
   \>_{\H_\rho^{\ell}} \\
&\quad+\eps \RE \<a^\eps,i\Delta a^\eps\>_{\H_\rho^\ell}.
   \end{align*}
Remark
that
\begin{equation*}
  \RE \<a^\eps,i\Delta a^\eps\>_{\H_\rho^\ell}=0,
\end{equation*}
so the Laplacian term is not present in energy estimates, which are
therefore independent of $\eps$. Like before, Cauchy-Schwarz
inequality yields
\begin{equation*}
  |\<a^\eps, \nabla\varphi
   \cdot \nabla a^\eps\>_{\H_\rho^{\ell}}|\le \|a^\eps\|_{\H_\rho^{\ell+1/2}}\|\nabla\varphi
   \cdot \nabla a^\eps\|_{\H_\rho^{\ell-1/2}}.
\end{equation*}
The last term is estimated thanks to the first point in
Lemma~\ref{lem:tame}, with
\begin{equation*}
  m=\ell-\frac{1}{2},\quad k=\ell+\frac{1}{2},\quad s= \ell+1.
\end{equation*}
Similarly,
\begin{equation*}
  |\<a^\eps, a^\eps\Delta\varphi
   \>_{\H_\rho^{\ell}} |\le \|a^\eps\|_{\H_\rho^{\ell+1/2}}\|
   a^\eps\Delta \varphi\|_{\H_\rho^{\ell-1/2}},
\end{equation*}
and the last term  is estimated thanks to the second point in
Lemma~\ref{lem:tame}, with
\begin{equation*}
  m=\ell-\frac{1}{2},\quad k=\ell,\quad s= \ell+\frac{3}{2}.
\end{equation*}
The lemma follows easily.
 \end{proof}

\section{A fundamental estimate}
\label{sec:fundamental}

In the framework of Theorem~\ref{theo:main}, the initial datum
$u^\eps_{\mid t=0}=a_0
e^{i\phi_0/\eps}$ belongs
to $H^\infty$, so the existence of 
$T^\eps>0$ (depending a priori on $\eps$),  and of a unique solution
$u^\eps \in C([0,T^\eps],H^\infty)$ 
to \eqref{eq:nls}-\eqref{eq:ci-u}, stems from standard theory (see
e.g. \cite{CazHar}). The fact that the existence time may be chosen
independent of $\eps$, along with the rest of the first point of
Theorem~\ref{theo:main}, stems from Proposition~\ref{prop:local} below.
\smallbreak

For a decreasing function $\rho$, we introduce the norm defined by 
\begin{equation}\label{eq:triple}
  \vvvert \psi\vvvert_{\ell,t}^2 = \max\left(\sup_{0\le s\le
    t}\|\psi(s)\|_{\H_{\rho(s)}^\ell} ^2,2\int_0^t |\dot \rho(s)|
  \|\psi(s)\|_{\H_{\rho(s)}^{\ell+1/2}}^2 ds\right). 
\end{equation}

\begin{proposition}\label{prop:local} 
  Let $\lambda\in \R$, $\ell>d/2+1$, $M_0>0$
  and $(\phi_0,a_0)\in \H_{M_0}^{\ell+1}\times \H_{M_0}^{\ell}$.\\
$1.$ There exists $M\gg 1$ such that if $T<M_0/M$ and $\rho(t) =M_0-Mt$,
\eqref{eq:syst-grenier}--\eqref{eq:ci} has a unique solution
\begin{equation*}
(\phi^\eps,a^\eps)\in 
C([0,T],\H_\rho^{\ell+1}\times \H_\rho^{\ell})\cap
L^2([0,T], \H_\rho^{\ell+3/2}\times \H_\rho^{\ell+1/2}),
\end{equation*}
with
\begin{equation}\label{eq:bound}
\vvvert  \phi^\eps\vvvert_{\ell+1,T}^2 \le 2\|\phi_0\|_{\H_{M_0}^{\ell+1}}^2+\| a_0\|_{\H_{M_0}^\ell}^{4\sigma},\quad\vvvert
  a^\eps\vvvert_{\ell,T}^2 \le 2\|a_0\|_{\H_{M_0}^{\ell}}^2.
\end{equation}
$2.$ If $R>0$ and $(\phi_0,a_0), (\varphi_0,b_0)\in
  \H_{M_0}^{\ell+1}\times \H_{M_0}^{\ell}$, with
  \begin{equation*}
    \|\phi_0\|_{\H_{M_0}^{\ell+1}}+\|a_0\|_{\H_{M_0}^\ell}\le R,\quad
    \|\varphi_0\|_{\H_{M_0}^{\ell+1}}+\|b_0\|_{\H_{M_0}^\ell}\le R ,
  \end{equation*}
there exists $K=K(R)$ such that if $M$ is chosen sufficiently large such that according to the first part of the proposition, \eqref{eq:syst-grenier}--\eqref{eq:ci} has solutions $(\phi^\eps,a^\eps)$ and $(\varphi^\eps,b^\eps)$ on $[0,T]$ corresponding respectively to the initial data $(\phi_0,a_0)$ and $(\varphi_0,b_0)$ (with the same choice of $\rho$ and the same assumption $T<M_0/M$), then 
\begin{equation*}
   \vvvert
\phi^\eps-\varphi^\eps\vvvert_{\ell+1,T}+\vvvert
a^\eps - b^\eps\vvvert_{\ell,T} \le K\(
\|\phi_0-\varphi_0\|_{\H_{M_0}^{\ell+1}} + \|a_0-b_0\|_{\H_{M_0}^{\ell}}\).
\end{equation*}

\end{proposition}
\begin{remark}
  The proof yields a rather implicit dependence of $M$ upon $M_0$ and $(\phi_0,a_0)$. As
  a consequence, it is not clear how to choose the best possible $T$,
  even for initial data whose Fourier transform is compactly
  supported. For our present concern, the important information is
  that we get some positive $T$ independent of $\eps$. 
\end{remark}
\begin{proof}
 To construct the solution, we resume the standard scheme from
 hyperbolic symmetric systems (see e.g. \cite{AlGe07}), that is, we
 consider the iterative  scheme defined by 
\begin{equation}\label{eq:iteration}
  \left\{
    \begin{aligned}
      &\d_t \phi^\eps_{j+1} + \frac{1}{2}\nabla \phi^\eps_j\cdot
      \nabla\phi^\eps_{j+1} +f(|a_j^\eps|^2) =0  
      ,\quad \phi^\eps_{j+1\mid t=0}=\phi_0,\\
& \d_t a^\eps_{j+1} + \nabla \phi^\eps_j \cdot \nabla a^\eps_{j+1}
+\frac{1}{2}a^\eps_{j+1}\Delta \phi^\eps_j=
\frac{i\eps}{2}\Delta a^\eps_{j+1},\quad a^\eps_{j+1\mid t=0}=a_0,
    \end{aligned}
\right.
\end{equation}
with $f(|a|^2)=\lambda |a|^{2\si}$, initialized with $(\phi_0^\eps,a_0^\eps)(t)=(\phi_0,a_0)$. For functions at the level of regularity of the norm \eqref{eq:triple} with $\ell>d/2$, the above
scheme is well defined: if $\vvvert \phi_j^\eps\vvvert_{\ell+1,T}+\vvvert
a_j^\eps\vvvert_{\ell,T}$ is finite, then $\phi_{j+1}^\eps$ and
$a_{j+1}^\eps$ are well-defined. Indeed, in the first equation, $\phi_{j+1}^\eps$
solves a linear transport equation with smooth coefficients, and the
second equation is equivalent to the linear Schr\"odinger equation
\begin{equation*}
  i\eps\d_t v_{j+1}^\eps +\frac{\eps^2}{2}\Delta v_{j+1}^\eps =
  -\(\d_t \phi_j^\eps +\frac{1}{2}|\nabla
  \phi_j^\eps|^2\)v_{j+1}^\eps,\quad v_{j+1\mid t=0}^\eps= a_0 e^{i\phi_j^\eps(0)/\eps},
\end{equation*}
through the relation $v_{j+1}^\eps = a_{j+1}^\eps
e^{i\phi_j^\eps/\eps}$. This is a linear Schr\"odinger equation with a
smooth and bounded external time-dependent potential, for which the
existence of an $L^2$-solution is granted. 
\smallbreak

The proof of the first part of the proposition goes in two steps: first, we prove
that the sequence $(\vvvert \phi_j^\eps\vvvert_{\ell+1,T}+\vvvert
a_j^\eps\vvvert_{\ell,T})_{j\ge 0}$ is bounded for some $T>0$ sufficiently small,
but independent of $\eps$. Then we show that
up to decreasing $T$, the series  
$$\sum_{j\ge 0} \( \vvvert
\phi_{j+1}^\eps-\phi_j^\eps\vvvert_{\ell+1,T}+\vvvert
a_{j+1}^\eps - a_j^\eps\vvvert_{\ell,T}\)$$ 
is converging. Note that unlike in the case of
hyperbolic symmetric systems in Sobolev spaces, the regularity is the
same at the two steps of the proof (in Sobolev spaces, the standard
proof involves first a bound in the large norm, then convergence in
the small norm). 
\smallbreak

\noindent {\bf First step: the sequence is bounded.} By integration,
Lemma~\ref{lem:evol-norm} yields, for a decreasing $\rho(t)$ and $T>0$
to be chosen later,
%\begin{align*}
%  \vvvert  \phi_{j+1}^\eps\vvvert_{\ell+1,T}^2 &\le \vvvert
%  \phi_{j+1}^\eps\vvvert_{\ell+1,0}^2 +C \int_0^T
%  \|\phi_{j+1}^\eps(t)\|_{\H_{\rho(t)}^{\ell+3/2}}^2
%  \|\phi_j^\eps(t)\|_{\H_{\rho(t)}^{\ell +1}}dt\\
%&\qquad 
%+C \int_0^T
%\|\phi_{j+1}^\eps(t)\|_{\H_{\rho(t)}^{\ell+3/2}} \|\phi_{j}^\eps(t)\|%_{\H_{\rho(t)}^{\ell+3/2}}
%\|\phi_{j+1}^\eps(t)\|_{\H_{\rho(t)}^{\ell+1}}dt\\
%&\qquad 
%+C \int_0^T\|a_{j}^\eps(t)\|_{\H_{\rho(t)}^{\ell+1/2}}
%\|\phi_{j+1}^\eps(t)\|_{\H_{\rho(t)}^{\ell+3/2}} 
%\|a_{j}^\eps(t)\|_{\H_{\rho(t)}^{\ell}}^{2\si-1}dt,\\
%  \vvvert  a_{j+1}^\eps\vvvert_{\ell,T}^2 &\le \vvvert
% a_{j+1}^\eps\vvvert_{\ell,0}^2 +C \int_0^T
%  \|a_{j+1}^\eps(t)\|_{\H_{\rho(t)}^{\ell+1/2}}^2
%  \|\phi_j^\eps(t)\|_{\H_{\rho(t)}^{\ell+1}}dt\\
%&\qquad 
%+C \int_0^T\|a_{j+1}^\eps(t)\|_{\H_{\rho(t)}^{\ell+1/2}}
%\|\phi_{j}^\eps(t)\|_{\H_{\rho(t)}^{\ell+3/2}} 
%\|a_{j+1}^\eps(t)\|_{\H_{\rho(t)}^{\ell}}dt.
%\end{align*}
%In view of the initial data, we infer
\begin{align*}
  \vvvert  \phi_{j+1}^\eps\vvvert_{\ell+1,T}^2 &\le \|
  \phi_0\|_{\H_{\rho(0)}^{\ell+1}}^2 +C \int_0^T
 \frac{1}{|\dot \rho(t)|} |\dot \rho(t)|\|\phi_{j+1}^\eps(t)\|_{\H_{\rho(t)}^{\ell+3/2}}^2
  \|\phi_j^\eps(t)\|_{\H_{\rho(t)}^{\ell +1}}dt\\
&\qquad +C \int_0^T \frac{1}{|\dot \rho(t)|} |\dot \rho(t)|
  \|\phi_{j+1}^\eps(t)\|_{\H_{\rho(t)}^{\ell+3/2}}
  \|\phi_{j}^\eps(t)\|_{\H_{\rho(t)}^{\ell+3/2}} 
\|\phi_{j+1}^\eps(t)\|_{\H_{\rho(t)}^{\ell+1}}dt\\
&\qquad 
+C \int_0^T \frac{1}{|\dot \rho(t)|} |\dot \rho(t)|\|a_{j}^\eps(t)\|_{\H_{\rho(t)}^{\ell+1/2}}
\|\phi_{j+1}^\eps(t)\|_{\H_{\rho(t)}^{\ell+3/2}} 
\|a_{j}^\eps(t)\|_{\H_{\rho(t)}^{\ell}}^{2\si-1}dt,\\
  \vvvert  a_{j+1}^\eps\vvvert_{\ell,T}^2 &\le \|
 a_0\|_{\H_{\rho(0)}^\ell}^2 +C \int_0^T
   \frac{1}{|\dot \rho(t)|} |\dot \rho(t)|\|a_{j+1}^\eps(t)\|_{\H_{\rho(t)}^{\ell+1/2}}^2
  \|\phi_j^\eps(t)\|_{\H_{\rho(t)}^{\ell+1}}dt\\
&\qquad 
+C \int_0^T \frac{1}{|\dot \rho(t)|} |\dot \rho(t)|\|a_{j+1}^\eps(t)\|_{\H_{\rho(t)}^{\ell+1/2}}
\|\phi_{j}^\eps(t)\|_{\H_{\rho(t)}^{\ell+3/2}} 
\|a_{j+1}^\eps(t)\|_{\H_{\rho(t)}^{\ell}}dt.
\end{align*}
H\"older and Cauchy-Schwarz inequalities yield
\begin{align*}
  \vvvert  \phi_{j+1}^\eps\vvvert_{\ell+1,T}^2 &\le \|
  \phi_0\|_{\H_{\rho(0)}^{\ell+1}}^2 +C \(\sup_{0\le t\le T}  
 \frac{1}{|\dot   \rho(t)|}  \)
\vvvert \phi_{j+1}^\eps\vvvert_{\ell+1,T}^2 
\vvvert \phi_{j}^\eps\vvvert_{\ell+1,T}
\\
&\qquad 
+C\( \sup_{0\le t\le T}  
 \frac{1}{|\dot   \rho(t)|} \)  
\vvvert \phi_{j+1}^\eps\vvvert_{\ell+1,T}
\vvvert a_{j}^\eps\vvvert_{\ell,T}^{2\si},\\
  \vvvert  a_{j+1}^\eps\vvvert_{\ell,T}^2 &\le \|
 a_0\|_{\H_{\rho(0)}^\ell}^2 +C \(\sup_{0\le t\le T}  
 \frac{1}{|\dot   \rho(t)|}   \)
\vvvert a_{j+1}^\eps\vvvert_{\ell,T}^2 
\vvvert \phi_{j}^\eps\vvvert_{\ell+1,T}.
\end{align*}
Recall that $M_0>0$ is given. Take $\phi_0\in \H_{M_0}^{\ell+1}$,
$a_0\in \H_{M_0}^{\ell}$ and
set $\rho(t)=M_0-Mt$. Under the condition
\begin{equation}\label{eq:rec}
\frac{C}{M}\vvvert\phi_j^\eps\vvvert_{\ell+1,T}\le\frac{1}{4},
\end{equation}
the previous inequalities imply
\begin{align*}
\frac{1}{2}\vvvert\phi_{j+1}^\eps\vvvert_{\ell+1,T}^2&\le\|\phi_0\|_{\H_{M_0}^{\ell+1}}^2+\frac{C^2}{M^2}\vvvert a_j^\eps\vvvert_{\ell,T}^{4\sigma},\\
\frac{3}{4}\vvvert a_{j+1}^\eps\vvvert_{\ell,T}^2&\le\|a_0\|_{\H_{M_0}^{\ell}}^2.
\end{align*}
Let us now choose $M=|\dot \rho(t)|$ is
sufficiently large  such that \eqref{eq:rec} holds for $j=0$ and such that 
\begin{align*}
&2\|\phi_0\|_{\H_{M_0}^{\ell+1}}^2+\frac{2C^2}{M^2}\(\frac{4}{3}\vvvert a_0^\eps\vvvert_{\ell,T}^2\)^{2\sigma}\le \frac{M^2}{16C^2},\\
&\frac{2C^2}{M^2}\(\frac{4}{3}\)^{2\sigma}\le 1.
\end{align*}
Note that in view of \eqref{eq:triple}, for all $T<M_0/M$ (so that $\rho$ remains positive on
$[0,T]$),
$$\vvvert a_0^\eps\vvvert_{\ell,T}^2=\max\(\|a_0\|_{\H_{M_0}^{\ell}}^2,\int\<\xi\>^{2\ell}e^{2M_0\<\xi\>}|\hat{a_0}(\xi)|^2\int_0^T 2M\<\xi\>e^{-2Mt\<\xi\>}dtd\xi \)=\|a_0\|_{\H_{M_0}^{\ell}}^2,$$
and similarly
$$\vvvert \phi_0^\eps\vvvert_{\ell+1,T}^2=\|\phi_0\|_{\H_{M_0}^{\ell+1}}^2,$$
so that our constraint on $M$ only depends on
$\|\phi_0\|_{\H_{M_0}^{\ell+1}}$ and $\|a_0\|_{\H_{M_0}^{\ell}}$. 
 Then, for $T<M_0/M$, the above inequalities yield, by induction, for all $j\ge 1$,
\begin{align*}
\vvvert\phi_{j}^\eps\vvvert_{\ell+1,T}^2&\le 2\|\phi_0\|_{\H_{M_0}^{\ell+1}}^2+\frac{2C^2}{M^2}\(\frac{4}{3}\| a_0\|^2_{\H_{M_0}^\ell}\)^{2\sigma}\le 2\|\phi_0\|_{\H_{M_0}^{\ell+1}}^2+\| a_0\|_{\H_{M_0}^\ell}^{4\sigma},\\
\vvvert a_{j}^\eps\vvvert_{\ell,T}^2&\le \frac{4}{3}\|a_0\|_{\H_{M_0}^{\ell}}^2.
\end{align*}
\smallbreak

\noindent {\bf Second step: the sequence converges.} For $j\ge
1$, consider the difference of two successive iterates: in the case of
the phase, we have
\begin{equation*}
  \d_t (\phi_{j+1}^\eps - \phi_{j}^\eps ) + \frac{1}{2}\(\nabla \phi^\eps_j\cdot
      \nabla\phi^\eps_{j+1} - \nabla \phi^\eps_{j-1}\cdot
      \nabla\phi^\eps_{j}\)
  +f(|a_j^\eps|^2)-f(|a_{j-1}^\eps|^2)=0, 
\end{equation*}
along with zero initial data. Inserting the term $|\nabla
\phi_j^\eps|^2$, and denoting by $\delta\phi_{j+1}^\eps =
\phi_{j+1}^\eps - \phi_{j}^\eps $, we can rewrite the above equation as
\begin{equation*}
  \d_t \delta\phi_{j+1}^\eps+ \frac{1}{2}\(\nabla \phi^\eps_j\cdot
      \nabla \delta \phi^\eps_{j+1}  + \nabla \delta \phi^\eps_j\cdot
      \nabla\phi^\eps_{j}\) +f(|a_j^\eps|^2)-f(|a_{j-1}^\eps|^2)=0.
\end{equation*}
\eqref{eq:evol-norm} yields, along with Cauchy-Schwarz inequality as
in the first step of the proof of Proposition~\ref{prop:local}:
\begin{align*}
   \vvvert  \delta\phi_{j+1}^\eps\vvvert_{\ell+1,T}^2 &\le \int_0^T  \|\delta
    \phi_{j+1}^\eps(t)\|_{\H_{\rho(t)}^{\ell+3/2}}   \| \nabla \phi^\eps_j\cdot
      \nabla \delta \phi^\eps_{j+1} \|_{\H_{\rho(t)}^{\ell+1/2}}dt\\
&\quad + \int_0^T \|\delta
    \phi_{j+1}^\eps(t)\|_{\H_{\rho(t)}^{\ell+3/2}}   \| \nabla \delta \phi^\eps_j\cdot
      \nabla\phi^\eps_{j}\|_{\H_{\rho(t)}^{\ell+1/2}} dt\\
&\quad +2\int_0^T  \|\delta
    \phi_{j+1}^\eps(t)\|_{\H_{\rho(t)}^{\ell+3/2}}
  \|f(|a_j^\eps|^2)-f(|a_{j-1}^\eps|^2) \|_{\H_{\rho(t)}^{\ell+1/2}}
  dt. 
\end{align*}
The first two terms are estimated thanks to the last point in
Lemma~\ref{lem:tame}, as in \eqref{eq:produit-grad}.
\smallbreak

Since $f(|z|^2)$ is a polynomial in $(z,\bar z)$,  the last
 point of Lemma~\ref{lem:tame} yields
 \begin{align*}
   &\|f(|a_j^\eps|^2)-  f(|a_{j-1}^\eps|^2)\|_{\H_\rho^{\ell+1/2}}\le C
   \(\|a_j^\eps\|_{\H_\rho^{\ell}}^{2\si-2} +
   \|a_{j-1}^\eps\|_{\H_\rho^{\ell}}^{2\si-2} \) \times \\
\qquad &\times\(
   \(\|a_j^\eps\|_{\H_\rho^{\ell}}+
   \|a_{j-1}^\eps\|_{\H_\rho^{\ell}}\)\|\delta
   a_j^\eps\|_{\H_\rho^{\ell+1/2}} + \(\|a_j^\eps\|_{\H_\rho^{\ell+1/2}}+
   \|a_{j-1}^\eps\|_{\H_\rho^{\ell+1/2}}\)\|\delta
   a_j^\eps\|_{\H_\rho^{\ell}}\) .
 \end{align*}
We conclude:
\begin{equation*}
    \vvvert  \delta\phi_{j+1}^\eps\vvvert_{\ell+1,T}^2  \le 
\frac{K}{M}\(  \vvvert  \delta\phi_{j+1}^\eps\vvvert_{\ell+1,T}^2 
+\vvvert  \delta\phi_{j}^\eps\vvvert_{\ell+1,T}^2 +
 \vvvert  \delta a_{j}^\eps\vvvert_{\ell,T}^2 \),
\end{equation*}
where $K$ stems from the first step. For $M$ sufficiently large,
\begin{equation*}
    \vvvert  \delta\phi_{j+1}^\eps\vvvert_{\ell+1,T}^2  \le 
\frac{2K}{M}\(  \vvvert  \delta\phi_{j}^\eps\vvvert_{\ell+1,T}^2 +
 \vvvert  \delta a_{j}^\eps\vvvert_{\ell,t}^2 \).
\end{equation*}
Similarly,  $\delta a_{j+1}^\eps$ solves
\begin{align*}
  \d_t \delta a_{j+1}^\eps + \nabla \phi_j^\eps\cdot \nabla \delta
  a_{j+1}^\eps + \nabla \delta \phi_j^\eps\cdot \nabla a_j^\eps +
  \frac{1}{2}\delta a_{j+1}^\eps\Delta \phi_j^\eps +
  \frac{1}{2}a_j^\eps\Delta \delta \phi_j^\eps = i\frac{\eps}{2}\Delta
  \delta a_{j+1}^\eps. 
\end{align*}
The last term is skew-symmetric, and thus does not appear in energy
estimates. Resuming the same estimates as in the proof of
Lemma~\ref{lem:evol-norm}, we come up with:
\begin{equation*}
    \vvvert  \delta a_{j+1}^\eps\vvvert_{\ell,T}^2  \le 
\frac{K}{M}\(  \vvvert  \delta a_{j+1}^\eps\vvvert_{\ell,T}^2 
+\vvvert  \delta\phi_{j}^\eps\vvvert_{\ell+1,T}^2 +
 \vvvert  \delta a_{j}^\eps\vvvert_{\ell,T}^2 \),
\end{equation*}
hence
\begin{equation*}
    \vvvert  \delta a_{j+1}^\eps\vvvert_{\ell,T}^2  \le 
\frac{2K}{M}\(  
\vvvert  \delta\phi_{j}^\eps\vvvert_{\ell+1,T}^2 +
 \vvvert  \delta a_{j}^\eps\vvvert_{\ell,T}^2 \),
\end{equation*}
up to increasing $M$ (hence decreasing $T$). For $M$ possibly even
larger, we infer that the series 
\begin{equation*}
  \sum_{j\ge 0} \( \vvvert
\phi_{j+1}^\eps-\phi_j^\eps\vvvert_{\ell+1,T}+\vvvert
a_{j+1}^\eps - a_j^\eps\vvvert_{\ell,T}\)
\end{equation*}
converges geometrically.
Uniqueness is a direct consequence of the estimates used in this
second step. \eqref{eq:bound} is obtained by letting $j$ go to infinity in the estimates at the end of the first step.
\smallbreak

The Lipschitzean property of the flow follows from calculations similar to those of the second step of the proof.  
\end{proof}

\section{Bounds on the numerical solution}
\label{sec:bounds}

\begin{proposition}\label{prop:evol-syst}
  Let $\lambda\in \R$, $\si\in \N$, and let $(\phi^\eps,a^\eps)$ be the
  solution of either of the systems \eqref{eq:syst-grenier},
  \eqref{eq:linear-syst}, or \eqref{eq:Y-syst}, with the notation
  $f(|z|^2)=\lambda |z|^{2\si}$. Let $s>d/2+1$, $\mu>0$, and $\ell\ge
  s$. Suppose that
  $(\phi^\eps,a^\eps)$  satisfies
  \begin{equation*}
   (\phi^\eps,a^\eps) \in 
C([0,T],\H_\rho^{s+1}\times \H_\rho^{s}),
  \end{equation*}
where $\rho(t)=M_0-Mt$ and $0<T<M_0/M$, with
\begin{equation*}
 \sup_{t\in [0,T]}\|\phi^\eps(t)\|_{\H_{\rho(t)}^{s+1}}+
\sup_{t\in [0,T]}\|a^\eps(t)\|_{\H_{\rho(t)}^{s}}\le \mu. 
\end{equation*}
Then, up to increasing $M$ (and decreasing $T$), 
\begin{equation*}
  \|\phi^\eps(t)\|_{\H_{\rho(t)}^{\ell+1}} +
  \|a^\eps(t)\|_{\H_{\rho(t)}^{\ell}}\le  \|\phi^\eps(0)\|_{\H_{M_0}^{\ell+1}} +
  \|a^\eps(0)\|_{\H_{M_0}^{\ell}},\quad \forall t\in [0,T].
\end{equation*}
\end{proposition}
Note that the assumption carries over a regularity at level $s>d/2+1$,
while the conclusion addresses the regularity at level $\ell\ge s$:
the above proposition may be viewed as a tame estimate result. 
\begin{proof}
  First, remark that $\vvvert \phi^\eps(T)\vvvert_{s+1}+\vvvert
a^\eps(T)\vvvert_{s}$ is a non-increasing function of $M$, provided that the
constraint $T<M_0/M$ remains fulfilled. 
\smallbreak

Second, note that it suffices to establish the result in the case of
\eqref{eq:syst-grenier}, since the other systems contain fewer terms,
and we will estimate each term present in \eqref{eq:syst-grenier}. 
\smallbreak

The idea of the result is then to view \eqref{eq:evol-norm} as a
parabolic estimate, with diffusive coefficient $-\dot \rho=M$. Indeed,
like in the proof of Lemma~\ref{lem:evol-norm}, we have
\begin{align*}
  & \d_t \|\phi^\eps\|_{\H_\rho^{\ell+1}}^2 + 2M
   \|\phi^\eps\|_{\H_\rho^{\ell+3/2}}^2\le
  C\|\phi^\eps\|_{\H_{\rho}^{\ell+3/2}}\( \|\nabla\phi^\eps\cdot \nabla
  \phi^\eps\|_{\H_{\rho}^{\ell+1/2}} + \|
    |a^\eps|^{2\si}\|_{\H_{\rho}^{\ell+1/2}} \),\\
& \d_t \|a^\eps\|_{\H_\rho^{\ell}}^2 + 2M
   \|a^\eps\|_{\H_\rho^{\ell+1/2}}^2 \le
  C\|a^\eps\|_{\H_{\rho}^{\ell+1/2}}\( \|\nabla\phi^\eps\cdot \nabla
  a^\eps\|_{\H_{\rho}^{\ell-1/2}} + \|
    a^\eps\Delta \phi^\eps\|_{\H_{\rho}^{\ell-1/2}} \).
\end{align*}
We then invoke Lemma~\ref{lem:tame} once more. Since the first two
points in Lemma~\ref{lem:tame} involve the constraint $k,s\ge m$, we
can rely only on \eqref{eq:tame}. We have
\begin{align*}
  \|\nabla\phi^\eps\cdot \nabla
  \phi^\eps\|_{\H_{\rho}^{\ell+1/2}} \le C \|\nabla
  \phi^\eps\|_{\H_{\rho}^{\ell+1/2}} \|\nabla
  \phi^\eps\|_{\H_{\rho}^{s-1}} \le C \|
  \phi^\eps\|_{\H_{\rho}^{\ell+3/2}} \| \phi^\eps\|_{\H_{\rho}^{s}}  ,
\end{align*}
since $s>d/2+1$. We have already used the estimate
\begin{equation*}
  \|  |a^\eps|^{2\si}\|_{\H_{\rho}^{\ell+1/2}} \le C
    \|a^\eps\|_{\H_\rho^s}^{2\si-1}\|a^\eps\|_{\H_{\rho}^{\ell+1/2}},
\end{equation*}
so that Young inequality yields
\begin{equation*}
  \d_t \|\phi^\eps\|_{\H_\rho^{\ell+1}}^2 + 2M
   \|\phi^\eps\|_{\H_\rho^{\ell+3/2}}^2\le
  C \(\mu +\mu^{2\si-1}\)\(\|
  \phi^\eps\|_{\H_{\rho}^{\ell+3/2}}^2 +
  \|a^\eps\|_{\H_{\rho}^{\ell+1/2}}^2\). 
\end{equation*}
Again, \eqref{eq:tame} yields
\begin{align*}
  \|\nabla\phi^\eps\cdot \nabla
  a^\eps\|_{\H_{\rho}^{\ell-1/2}}  &\le C\(
  \|\nabla\phi^\eps\|_{\H_{\rho}^{\ell-1/2}}
\| \nabla 
  a^\eps\|_{\H_{\rho}^{s-1}}+\|\nabla\phi^\eps\|_{\H_{\rho}^{s-1}}
\|\nabla
  a^\eps\|_{\H_{\rho}^{\ell-1/2}}\)\\
&\le C\(
  \|\phi^\eps\|_{\H_{\rho}^{\ell+1/2}}
\|  a^\eps\|_{\H_{\rho}^{s}}+\|\phi^\eps\|_{\H_{\rho}^{s}}
\|
  a^\eps\|_{\H_{\rho}^{\ell+1/2}}\),
\end{align*}
and
\begin{align*}
   \|a^\eps\Delta\phi^\eps\|_{\H_{\rho}^{\ell-1/2}}  &\le C\(
  \|\Delta\phi^\eps\|_{\H_{\rho}^{\ell-1/2}}
\| 
  a^\eps\|_{\H_{\rho}^{s-1}}+\|\Delta\phi^\eps\|_{\H_{\rho}^{s-1}}
\| a^\eps\|_{\H_{\rho}^{\ell-1/2}}\)\\
&\le C\(
  \|\phi^\eps\|_{\H_{\rho}^{\ell+3/2}}
\|  a^\eps\|_{\H_{\rho}^{s}}+\|\phi^\eps\|_{\H_{\rho}^{s+1}}
\|
  a^\eps\|_{\H_{\rho}^{\ell+1/2}}\).
\end{align*}
We come up with
\begin{align*}
  \d_t \( \|\phi^\eps\|_{\H_\rho^{\ell+1}}^2
  +\|a^\eps\|_{\H_\rho^{\ell}}^2 \) +& 2M\( \|
  \phi^\eps\|_{\H_{\rho}^{\ell+3/2}}^2 +
  \|a^\eps\|_{\H_{\rho}^{\ell+1/2}}^2\) \\
&\le C\(\mu +\mu^{2\si-1}\) \(\|
  \phi^\eps\|_{\H_{\rho}^{\ell+3/2}}^2 +
  \|a^\eps\|_{\H_{\rho}^{\ell+1/2}}^2\).
\end{align*}
Choosing $2M \ge C\(\mu +\mu^{2\si-1}\)$ thus yields the result. 
\end{proof}
We readily infer:
\begin{corollary}\label{cor:borneZ}
  Let $\ell\ge s>d/2+1$, and $\tau>0$. Suppose that the numerical
  solution
  \begin{equation*}
    {\mathcal Z}_\eps^t
\begin{pmatrix}
  \phi_0^\eps \\
a_0^\eps
\end{pmatrix} =
\begin{pmatrix}
  \phi_t^\eps\\
a_t^\eps
\end{pmatrix}
  \end{equation*}
satisfies
\begin{equation*}
 \sup_{t\in [0,\tau]}\|\phi^\eps_t\|_{\H_{\rho(t)}^{s+1}}+
\sup_{t\in [0,\tau]}\|a^\eps_t\|_{\H_{\rho(t)}^{s}}\le \mu,
\end{equation*}
where $\rho(t)=M_0-Mt$. Then, up to increasing $M$ (and possibly
decreasing $\tau$), 
\begin{equation*}
  \|\phi^\eps_t\|_{\H_{\rho(t)}^{\ell+1}} +
  \|a^\eps_t\|_{\H_{\rho(t)}^{\ell}}\le  \|\phi^\eps_0\|_{\H_{M_0}^{\ell+1}} +
  \|a^\eps_0\|_{\H_{M_0}^{\ell}},\quad \forall t\in [0,\tau].
\end{equation*}
\end{corollary}

\section{Local error estimate}
\label{sec:local-error}

We resume the computations from \cite{CaSINUM}, based on the general
formula established  in \cite{DeTh13}. For a
possibly nonlinear operator $A$, we denote by $\E_A$ the associated flow:
\begin{equation*}
  \d_t \E_A(t,v)=A\(\E_A(t,v)\);\quad \E_A(0,v)=v.
\end{equation*}

\begin{theorem}[Theorem~1 from \cite{DeTh13}]\label{theo:error}
 Suppose that $F(u)=A(u)+B(u)$, and denote by
 \begin{equation*}
   {\mathcal S}^t(u)= \E_F\(t,u\)\text{ and }{\mathcal Z}^t(u)
   =\E_B\(t,\E_A(t,u)\)
 \end{equation*}
the exact flow and the Lie-Trotter flow, respectively. 
 Let $\L (t,u) = {\mathcal Z}^t(u)-{\mathcal S}^t(u)$. We have the exact formula
  \begin{align*}
    \L(t,u) =\int_0^t
    \int_0^{\tau_1}\d_2\E_F& \(t-\tau_1,{\mathcal Z}^{\tau_1}(u)\) \d_2
    \E_B\(\tau_1-\tau_2,\E_A(\tau_1,u)\) \\
&\times [B,A]\(\E_B\(\tau_2,\E_A\(\tau_1,u\)\)\)d\tau_2d\tau_1,
  \end{align*}
where $[B,A](v)=B'(v)A(v)-A'(v)B(v)$.
\end{theorem}
In the case of the Lie-Trotter splitting \eqref{eq:Z} for equation \eqref{eq:nls}, we would have
\begin{equation*}
  A = i\frac{\eps}{2}\Delta;\quad B(v) =
  -\frac{i}{\eps}f(|v|^2)v,\quad f(|v|^2)=\lambda |v|^{2\si};\quad
  F(v)= A(v)+B(v),
\end{equation*}
where we have omitted the dependence upon $\eps$ in the notations for the
sake of brevity.
\smallbreak

However, as pointed out in \cite{CaSINUM}, using the above result directly in
terms of the wave function $u^\eps$ does not seem convenient. In the
context of WKB regime, we rather consider the operators $A$ 
and $B$ defined by 
\begin{equation}\label{eq:AB}
  A
  \begin{pmatrix}
    \phi\\
a
  \end{pmatrix}
=
\begin{pmatrix}
  -\frac{1}{2}|\nabla \phi|^2\\
-\nabla \phi\cdot \nabla a -\frac{1}{2}a\Delta \phi
+i\frac{\eps}{2}\Delta a
\end{pmatrix},
\quad B  \begin{pmatrix}
    \phi\\
a
  \end{pmatrix}
=
\begin{pmatrix}
  -f(|a|^2)\\
0
\end{pmatrix}.
\end{equation}
We note that with this approach, neither
$A$ nor $B$ is a linear operator. 

\begin{lemma}\label{lem:crochet}
  Let $A$ and $B$ defined by \eqref{eq:AB}. Their commutator is given by
\begin{equation*}
  [A,B]\begin{pmatrix}
    \phi\\
a
  \end{pmatrix} =
  \begin{pmatrix}
    \nabla \phi\cdot \nabla f\(|a|^2\) -\DIV\left(|a|^2\nabla\phi+\eps\IM(\overline{a}\nabla a)\right)f'(|a|^2)\\
\nabla a\cdot \nabla f\(|a|^2\) +\frac{1}{2}a\Delta f\(|a|^2\)
  \end{pmatrix}.
\end{equation*}
As a consequence, if $\ell>d/2+3$, $\rho>0$, 
$\| \phi\|_{\H^{\ell+1}_\rho}\le \mu$, $\|a\|_{\H^{\ell}_\rho}\le \mu$, then there exists
$C=C(\mu)$ independent of $\eps\in [0,1]$
such that
\begin{equation*}
 \begin{pmatrix}
    \varphi\\
b
  \end{pmatrix}= [A,B]\begin{pmatrix}
    \phi\\
a
  \end{pmatrix} 
   \quad \text{satisfies }
\left\{
  \begin{aligned}
    \|\varphi\|_{\H^{\ell-2}_\rho}&\le C
  \(\| \phi\|_{\H^{\ell+1}_\rho} + \|a\|_{\H^{\ell}_\rho}\),\\
 \|b\|_{\H^{\ell-3}_\rho}&\le C
  \|a\|_{\H^{\ell}_\rho}.
  \end{aligned}
\right.
\end{equation*}
\end{lemma}
\begin{proof} Like in \cite{CaSINUM}, we have
 \begin{equation*}
 A'\begin{pmatrix}\phi\\a  \end{pmatrix}\begin{pmatrix}\varphi\\b \end{pmatrix}=\begin{pmatrix}-\nabla\phi\cdot\nabla\varphi\\
 -\nabla\phi\cdot\nabla b-\nabla\varphi\cdot\nabla a-\frac{1}{2}b\Delta\phi-\frac{1}{2}a\Delta\varphi+i\frac{\eps}{2}\Delta b
 \end{pmatrix},
 \end{equation*}
 whereas unlike in \cite{CaSINUM}, we consider a function $f$ which is
 not necessarily linear, so that the linearized operator of $B$ is
 given by 
 \begin{equation*}
 B'\begin{pmatrix}\phi\\a  \end{pmatrix}\begin{pmatrix}\varphi\\b \end{pmatrix}=\begin{pmatrix}-2\RE(\overline{a}b)f'(|a|^2)\\ 0\end{pmatrix}
 \end{equation*}
 and thus
 \begin{equation*}
 B'\begin{pmatrix}\phi\\a  \end{pmatrix}\(A\begin{pmatrix}\phi\\a \end{pmatrix}\)=\begin{pmatrix}\(2\RE(\overline{a}\nabla a\cdot\nabla\phi)+|a|^2\Delta\phi+\eps\IM(\overline{a}\Delta a)\)f'(|a|^2)\\ 0\end{pmatrix}.
 \end{equation*}
 The  explicit formula for
$[A,B]$ follows as in \cite{CaSINUM}. The estimates then follow
directly from \eqref{eq:tame} and \eqref{eq:der}. 
\end{proof}
We have the explicit formula
\begin{equation}\label{eq:EB}
 {\mathcal Y}_\eps^t\begin{pmatrix}
    \phi\\
a
  \end{pmatrix}=  \E_B\(t, \begin{pmatrix}
    \phi\\
a
  \end{pmatrix}\) 
=
\begin{pmatrix}
  \phi- t f(|a|^2)\\
a
\end{pmatrix},
\end{equation}
and we readily infer
\begin{equation}\label{eq:d2EB}
   \d_2 \E_B\(t, \begin{pmatrix}
    \phi\\
a
  \end{pmatrix}\) \begin{pmatrix}
    \varphi\\
b
  \end{pmatrix}
=
\begin{pmatrix}
  \varphi- 2\si\lambda t|a|^{2\si-2}\RE (\overline a b)\\
b
\end{pmatrix}.
\end{equation}
Finally, we compute that
\begin{equation*}
\begin{pmatrix}
  \varphi(t)\\
b(t)
\end{pmatrix}=  \d_2\E_F\(t,\begin{pmatrix}
    \phi_{0}\\
a_{0}
  \end{pmatrix}
\)\begin{pmatrix}
    \varphi_0\\
b_0
  \end{pmatrix}
\end{equation*}
solves the system
\begin{equation}\label{eq:linearise}
  \left\{
    \begin{aligned}
      & \d_t \varphi + \nabla \phi\cdot \nabla \varphi +
    2\si\lambda |a|^{2\si-2}\RE (\overline a b) =0;\quad \varphi_{\mid t=0}=\varphi_0,\\
& \d_t b +\nabla \phi\cdot \nabla b+\nabla \varphi\cdot \nabla a +
\frac{1}{2}\( b\Delta \phi+a\Delta \varphi\)=i\frac{\eps}{2}\Delta
b;\quad b_{\mid t=0}=b_0, 
    \end{aligned}
\right.
\end{equation}
where
\begin{equation*}
\begin{pmatrix}
    \phi(t)\\
a(t)
  \end{pmatrix}=\E_F\(t,\begin{pmatrix}
    \phi_0\\
a_0
  \end{pmatrix}\).
\end{equation*}
\begin{lemma}\label{lem:flotlinearise}
  Let $\ell>d/2+1$, $s\ge \ell$ and $(\varphi_0, b_0)\in \H_{M_0}^{\ell+1}\times \H_{M_0}^{\ell}$. Assume that $(\phi,a)\in C([0,T],\H_\rho^{s+1}\times \H_\rho^{s})\cap L^2([0,T],\H_\rho^{s+3/2}\times \H_\rho^{s+1/2})$.   Then for $M$ sufficiently large and $T<M_0/M$, the
  solution to \eqref{eq:linearise} satisfies 
  \begin{equation*}
  \vvvert\varphi\vvvert_{\ell+1,T}^2+\vvvert b\vvvert_{\ell,T}^2\le 4\|\varphi_0\|_{\H_{M_0}^{\ell+1}}^2+4\|b_0\|_{\H_{M_0}^{\ell}}^2.
  \end{equation*}
\end{lemma}
\begin{proof}
The proof is quite similar to the one of Lemma \ref{lem:evol-norm} and Proposition \ref{prop:local}. We take the $\H_\rho^{\ell+1}$ scalar product of the first equation in \eqref{eq:linearise} with $\varphi$, and the $\H_\rho^{\ell}$ scalar product of the second one with $b$. We get
\begin{align*}
   \d_t \|\varphi\|_{\H_\rho^{\ell+1}}^2 + 2M
   \|\varphi\|_{\H_\rho^{\ell+3/2}}^2&\le
  C\|\varphi\|_{\H_{\rho}^{\ell+3/2}}\( \|\nabla\phi\cdot \nabla
  \varphi\|_{\H_{\rho}^{\ell+1/2}} + \|
    |a|^{2\si-2}\RE(\overline{a}b)\|_{\H_{\rho}^{\ell+1/2}} \),\\
 \d_t \|b\|_{\H_\rho^{\ell}}^2 + 2M
   \|b\|_{\H_\rho^{\ell+1/2}}^2 &\le
  C\|b\|_{\H_{\rho}^{\ell+1/2}}\( \|\nabla\phi\cdot \nabla
  b\|_{\H_{\rho}^{\ell-1/2}} + \|\nabla\varphi\cdot \nabla
  a\|_{\H_{\rho}^{\ell-1/2}}\right.\\
&\qquad  \left.+\|
    b\Delta \phi\|_{\H_{\rho}^{\ell-1/2}}+\|
    a\Delta \varphi\|_{\H_{\rho}^{\ell-1/2}} \).
\end{align*}
Then, the use of \eqref{eq:tame} with $m>d/2$ and integration in time
yield, with estimates similar to those presented in the proof of
Proposition~\ref{prop:local}, 
\begin{align*}
   \|\varphi(t)\|_{\H_{\rho(t)}^{\ell+1}}^2 + 2M
   \int_0^t\|\varphi(\tau)\|_{\H_{\rho(\tau)}^{\ell+3/2}}^2 &d\tau\le 
                                                                  \|\varphi_0\|_{\H_{M_0}^{\ell+1}}^2\\
+
  C&\int_0^t\|\varphi(\tau)\|_{\H_{\rho(\tau)}^{\ell+3/2}}
  \|\phi(\tau)\|_{\H_{\rho(\tau)}^{\ell+3/2}}\|\varphi(\tau)\|_{\H_{\rho(\tau)}^{m+1}}d\tau\\
+C&\int_0^t\|\varphi(\tau)\|_{\H_{\rho(\tau)}^{\ell+3/2}}^2
  \|\phi(\tau)\|_{\H_{\rho(\tau)}^{m+1}}d\tau\\
+C&\int_0^t\|\varphi(\tau)\|_{\H_{\rho(\tau)}^{\ell+3/2}}
  \|a(\tau)\|_{\H_{\rho(\tau)}^{\ell+1/2}}\|a(\tau)\|_{\H_{\rho(\tau)}^{m}}^{2\si-2}\|b(\tau)\|_{\H_{\rho(\tau)}^{m}}d\tau\\
  +C&\int_0^t\|\varphi(\tau)\|_{\H_{\rho(\tau)}^{\ell+3/2}}
  \|b(\tau)\|_{\H_{\rho(\tau)}^{\ell+1/2}}\|a(\tau)\|_{\H_{\rho(\tau)}^{m}}^{2\si-1}d\tau,
\end{align*}
\begin{align*}
  \|b(t)\|_{\H_{\rho(t)}^{\ell}}^2 + 2M\int_0^t
   \|b(\tau)\|_{\H_{\rho(\tau)}^{\ell+1/2}}^2& d\tau
                                               \le\|b_0\|_{\H_{M_0}^{\ell}}^2\\
&+
  C\int_0^t\|b(\tau)\|_{\H_{\rho(\tau)}^{\ell+1/2}}\|\phi(\tau)\|_{\H_{\rho(\tau)}^{\ell+1/2}}\|b(\tau)\|_{\H_{\rho(\tau)}^{m+1}}d\tau\\
  &+C\int_0^t\|b(\tau)\|_{\H_{\rho(\tau)}^{\ell+1/2}}^2\|\phi(\tau)\|_{\H_{\rho(\tau)}^{m+1}}d\tau\\
  &+  C\int_0^t\|b(\tau)\|_{\H_{\rho(\tau)}^{\ell+1/2}}\|\varphi(\tau)\|_{\H_{\rho(\tau)}^{\ell+1/2}}\|a(\tau)\|_{\H_{\rho(\tau)}^{m+1}}d\tau\\
   &+  C\int_0^t\|b(\tau)\|_{\H_{\rho(\tau)}^{\ell+1/2}}\|a(\tau)\|_{\H_{\rho(\tau)}^{\ell+1/2}}\|\varphi(\tau)\|_{\H_{\rho(\tau)}^{m+1}}d\tau\\
     &+  C\int_0^t\|b(\tau)\|_{\H_{\rho(\tau)}^{\ell+1/2}}\|b(\tau)\|_{\H_{\rho(\tau)}^{\ell-1/2}}\|\phi(\tau)\|_{\H_{\rho(\tau)}^{m+2}}d\tau\\
          &+  C\int_0^t\|b(\tau)\|_{\H_{\rho(\tau)}^{\ell+1/2}}\|\phi(\tau)\|_{\H_{\rho(\tau)}^{\ell+3/2}}\|b(\tau)\|_{\H_{\rho(\tau)}^{m}}d\tau\\
              &+  C\int_0^t\|b(\tau)\|_{\H_{\rho(\tau)}^{\ell+1/2}}\|a(\tau)\|_{\H_{\rho(\tau)}^{\ell-1/2}}\|\varphi(\tau)\|_{\H_{\rho(\tau)}^{m+2}}d\tau\\
          &+  C\int_0^t\|b(\tau)\|_{\H_{\rho(\tau)}^{\ell+1/2}}\|\varphi(\tau)\|_{\H_{\rho(\tau)}^{\ell+3/2}}\|a(\tau)\|_{\H_{\rho(\tau)}^{m}}d\tau\\ 
\end{align*}
We choose $m=\ell>d/2+1$ in the estimate for $\varphi$ and $m=\ell-1>d/2$ in the estimate for $b$. Denoting 
$$\mu= \vvvert\phi\vvvert_{\ell+1,T}+\vvvert a\vvvert_{\ell,T},$$
and 
$$\|\psi\|_{L^2_t\H_{\rho}^k}^2=\int_0^t\|\psi(\tau)\|_{\H_{\rho(\tau)}^k}^2d\tau,$$
since the $\H_\rho^k$ norms are increasing with $k$, Cauchy-Schwarz in time yields
\begin{align*}
  & \|\varphi(t)\|_{\H_{\rho(t)}^{\ell+1}}^2 +
  2M\|\varphi\|_{L^2_t\H_{\rho}^{\ell+3/2}}^2 \le
    \|\varphi_0\|_{\H_{M_0}^{\ell+1}}^2+
    C\|\varphi\|_{L^2_t\H_{\rho}^{\ell+3/2}} \times\\
 & \(\frac{\mu}{\sqrt{M}}\sup_{0\le\tau\le
   t}\|\varphi(\tau)\|_{\H_{\rho(\tau)}^{\ell+1}}+\mu\|\varphi\|_{L^2_t\H_{\rho}^{\ell+3/2}} 
+\frac{\mu^{2\sigma-1}}{\sqrt{M}}
\sup_{0\le\tau\le
   t}\|b(\tau)\|_{\H_{\rho(\tau)}^{\ell}}+\mu^{2\sigma-1}\|b\|_{L^2_t\H_{\rho}^{\ell+1/2}}\) ,
\end{align*}  
\begin{align*}  
&  \|b(t)\|_{\H_{\rho(t)}^{\ell}}^2+
  2M\|b\|_{L^2_t\H_{\rho}^{\ell+1/2}}^2 \le
  \|b_0\|_{\H_{M_0}^{\ell}}^2
  +C\mu\|b\|_{L^2_t\H_{\rho}^{l+1/2}}\times \\
&\left(\frac{1}{\sqrt{M}}\sup_{0\le\tau\le t}\|b(\tau)\|_{\H_{\rho(\tau)}^{l}}+\|b\|_{L^2_t\H_{\rho}^{l+1/2}}+\|\varphi\|_{L^2_t\H_{\rho}^{l+3/2}}+
  \frac{1}{\sqrt{M}}\sup_{0\le\tau\le t}\|\varphi(\tau)\|_{\H_{\rho(\tau)}^{l+1}}\).
\end{align*}
Adding the last two inequalities, we deduce, for $\mu\le C$,
\begin{align*}
  & \|\varphi(t)\|_{\H_{\rho(t)}^{\ell+1}}^2+\|b(t)\|_{\H_{\rho(t)}^{\ell}}^2+2M\|\varphi\|_{L^2_t\H_{\rho}^{\ell+3/2}}^2+2M\|b\|_{L^2_t\H_{\rho}^{\ell+1/2}}^2\le \|\varphi_0\|_{\H_{M_0}^{\ell+1}}^2+\|b_0\|_{\H_{M_0}^{\ell}}^2\\
   &\quad +C\frac{\mu}{M}\(\sup_{0\le\tau\le t}\|\varphi(\tau)\|_{\H_{\rho(\tau)}^{l+1}}^2+\sup_{0\le\tau\le t}\|b(\tau)\|_{\H_{\rho(\tau)}^{l}}^2+M\|\varphi\|_{L^2_t\H_{\rho}^{l+3/2}}^2+M\|b\|_{L^2_t\H_{\rho}^{l+1/2}}^2\),
\end{align*}
from which the inequality of the lemma easily follows provided $M$ is sufficiently large.
\end{proof}

We infer the WKB local error estimate:
\begin{theorem}[Local error estimate for WKB
  states]\label{theo:errlocBKW} 
  Let $\ell>d/2+5$, $M_0>0$, $M\gg1$, $\rho(t)=M_0-Mt$ and $\mu>0$. Let $(\phi_0,a_0)\in \H_{M_0}^{\ell+1}\times \H_{M_0}^{\ell}$ such that
  \begin{equation*}
    \|\phi_0\|_{\H_{M_0}^{\ell+1}}\le \mu,\quad \|a_0\|_{\H_{M_0}^{\ell}}\le \mu.
  \end{equation*}
There exist $C,c_0>0$ (depending on $\mu$) independent
  of $\eps\in (0,1]$ such that 
  \begin{equation*}
    \L\(t,
    \begin{pmatrix}
      \phi_0\\
a_0
    \end{pmatrix}\):= \mathcal Z^t_\eps\begin{pmatrix}
      \phi_0\\
a_0
    \end{pmatrix}-
\mathcal S^t_\eps\begin{pmatrix}
      \phi_0\\
a_0
    \end{pmatrix}  =
    \begin{pmatrix}
      \Psi^\eps(t)\\
A^\eps(t)
    \end{pmatrix},
  \end{equation*}
satisfies 

\begin{equation*}
  \|\Psi^\eps(t)\|_{\H_{\rho(t)}^{\ell-3}}+\|A^\eps(t)\|_{\H_{\rho(t)}^{\ell-4}}\le C
  t^2,\quad 0\le t\le c_0. 
\end{equation*}
\end{theorem}
The above result obviously involves a loss of regularity, between the
initial assumptions and the conclusion. It is important to note that
the local error estimate is used only once in the final Lady
Windermere's fan argument presented in the next section, so this loss
is not a serious problem.  
\begin{proof}
  Let $t\in [0,c_0]$, and fix $\tau_1,\tau_2$ such that $0\le \tau_2\le
  \tau_1\le t$. Introduce the following intermediary notations:
  \begin{align*}
    & \begin{pmatrix}
      \phi_1\\
a_1^\eps
    \end{pmatrix}=\E_A\(\tau_1,
    \begin{pmatrix}
      \phi_0\\
a_0
    \end{pmatrix}\)
   ,\\
&\begin{pmatrix}
      \phi_2^\eps\\
a_2^\eps
    \end{pmatrix}=\E_B\(\tau_2,
\begin{pmatrix}
      \phi_1\\
a_1^\eps
    \end{pmatrix}\) ,\quad\quad \begin{pmatrix}
\tilde \phi_2^\eps\\
\tilde a_2^\eps
    \end{pmatrix}=\E_B\(\tau_1,
\begin{pmatrix}
      \phi_1\\
a_1^\eps
    \end{pmatrix}\)\\
&\begin{pmatrix}
      \phi_3^\eps\\
a_3^\eps
    \end{pmatrix}=[B,A] \begin{pmatrix}
      \phi_2^\eps\\
a_2^\eps
    \end{pmatrix} ,\qquad
\begin{pmatrix}
      \phi_4^\eps\\
a_4^\eps
    \end{pmatrix}= \d_2 \E_B\(\tau_1-\tau_2, 
\begin{pmatrix}
  \phi_1\\
a_1^\eps
\end{pmatrix}\)
\begin{pmatrix}
      \phi_3^\eps\\
a_3^\eps
    \end{pmatrix}.
  \end{align*}
  Then in view of Theorem~\ref{theo:error}, we have
  \begin{equation*}
    \begin{pmatrix}
      \Psi^\eps(t)\\
A^\eps(t)
    \end{pmatrix}=
\int_0^t\int_0^{\tau_1}\d_2\E_F\(t-\tau_1,
\begin{pmatrix}
      \tilde \phi_2^\eps\\
\tilde a_2^\eps
    \end{pmatrix}\)
\begin{pmatrix}
      \phi_4^\eps\\
a_4^\eps
    \end{pmatrix}d\tau_2d\tau_1.
  \end{equation*}
Since $\ell>d/2+1$, Proposition~\ref{prop:local} for $\lambda=0$ ensures that $(\phi_1,a_1^\eps)\in \H_{\rho(\tau_1)}^{\ell+1}\times \H_{\rho(\tau_1)}^{\ell}$ is well defined provided $\tau_1\le c_0< M_0/M$, with  (according to \eqref{eq:bound} where we can remove the $\|a_0\|_{\H_{M_0}^\ell}^{4\sigma}$ term because $\lambda=0$)
\begin{equation*}
  \|\phi_1\|_{\H_{\rho(\tau_1)}^{\ell+1}}\le 2\mu,\quad \|a_1^\eps\|_{\H_{\rho(\tau_1)}^{\ell}}\le 2\mu. 
\end{equation*}
\eqref{eq:EB} writes $(\phi_2^\eps,a_2^\eps)=(\phi_1^\eps-\lambda\tau_2|a_1^\eps|^{2\sigma},a_1^\eps)$ and thus \eqref{eq:tame} yields (in the calculations below, the constant $C$ may depend on $\mu$ and may change from line to line)
\begin{equation*}
  \|\phi_2^\eps\|_{\H_{\rho(\tau_1)}^{\ell}}\le 2\mu+C\mu^{2\sigma}\le C\mu,\quad \|a_2^\eps\|_{\H_{\rho(\tau_1)}^{\ell}}\le 2\mu,
\end{equation*}
because $\ell>d/2$. Similarly,
\begin{equation}\label{eq:phi2tilde}
  \|\tilde\phi_2^\eps\|_{\H_{\rho(\tau_1)}^{\ell}}\le C\mu,\quad \|\tilde a_2^\eps\|_{\H_{\rho(\tau_1)}^{\ell}}\le 2\mu .
\end{equation}
Next, since $\ell-1>d/2+3$, Lemma~\ref{lem:crochet} implies
\begin{equation*}
  \|\phi_3^\eps\|_{\H_{\rho(\tau_1)}^{\ell-3}}\le C\mu,\quad \|a_3^\eps\|_{\H_{\rho(\tau_1)}^{\ell-4}}\le C\mu.
\end{equation*}
In view of \eqref{eq:d2EB}, we have
\begin{equation*}
\phi_4^\eps = \phi_3^\eps
-2\sigma\lambda (\tau_1-\tau_2)|a_1^\eps|^{2\sigma-2} \RE \(\overline a_1^\eps a_3^\eps\), \quad a_4^\eps=a_3^\eps,
\end{equation*}
and therefore
\begin{equation}\label{eq:phi4}
    \|\phi_4^\eps\|_{\H_{\rho(\tau_1)}^{\ell-3}}\le C\mu,\quad \|a_4^\eps\|_{\H_{\rho(\tau_1)}^{\ell-4}}\le C\mu,
\end{equation}
since $\ell-3>d/2$ and thanks to \eqref{eq:tame}. 

Finally, we prove that if $\ell>d/2+5$, the $\H_{\rho(t)}^{\ell-3}\times \H_{\rho(t)}^{\ell-4}$ norm of 
\begin{equation*}
\begin{pmatrix}
      \phi_5^\eps\\
a_5^\eps
    \end{pmatrix}= \d_2\E_F\(t-\tau_1,
\begin{pmatrix}
      \tilde \phi_2^\eps\\
\tilde a_2^\eps
    \end{pmatrix}\)
\begin{pmatrix}
      \phi_4^\eps\\
a_4^\eps
    \end{pmatrix}
  \end{equation*}
is uniformly bounded in $t,\tau_1,\tau_2$ as long as $0\le \tau_2\le
  \tau_1\le t\le T<M_0/M$. For this purpose, first note that since $\ell-1>d/2+1$, it follows from \eqref{eq:phi2tilde} and Proposition \ref{prop:local} that we can choose $M=M(\mu)$ sufficiently large such that if $0<T-\tau_1<\rho(\tau_1)/M$,
  \begin{align*}
\begin{pmatrix}\phi\\ a\end{pmatrix}(\tau)= \E_F\(\tau-\tau_1,
\begin{pmatrix}
      \tilde \phi_2^\eps\\
\tilde a_2^\eps
    \end{pmatrix}\) \text{ is such that}\\
\begin{pmatrix}\phi\\ a\end{pmatrix}\in
    C\([\tau_1,T],\H_{\rho}^\ell\times\H_{\rho}^{\ell-1}\)\cap
    L^2\([\tau_1,T],\H_{\rho}^{\ell+1/2}\times\H_{\rho}^{\ell-1/2}\), 
  \end{align*}
  with
  \begin{align*}
  \max&\left(\sup_{\tau_1\le \tau\le
    T}\|\phi(\tau)\|_{\H_{\rho(\tau)}^{\ell}} ^2,\sup_{\tau_1\le \tau\le
    T}\|a(\tau)\|_{\H_{\rho(\tau)}^{\ell-1}} ^2\right. ,\\
    & \left. 2M\int_{\tau_1}^T 
  \|\phi(\tau)\|_{\H_{\rho(\tau)}^{\ell+1/2}}^2 d\tau,2M\int_{\tau_1}^T 
  \|a(\tau)\|_{\H_{\rho(\tau)}^{\ell-1/2}}^2 d\tau\right)\le C(\mu+\mu^{2\sigma}). 
\end{align*}
(Note that $\rho(\tau)=\rho(\tau_1)-M(\tau-\tau_1)$). Then, thanks to
\eqref{eq:phi4} and Lemma \ref{lem:flotlinearise}, since
$\ell-4>d/2+1$ and $s=\ell-1\ge \ell-4$, choosing possibly $M=M(\mu)$
even larger,  
\begin{align*}
\max\(\|\phi_5^\eps\|_{\H_{\rho(t)}^{\ell-3}},\|a_5^\eps\|_{\H_{\rho(t)}^{\ell-4}}\)\le C\mu.
\end{align*}
The theorem follows.
\end{proof}
Back to the wave functions, we obtain an
estimate similar to the one presented in \cite[Section~4.2.2]{DeTh13}:
\begin{corollary}\label{cor:local-wave}
 Under the assumptions of Theorem \ref{theo:errlocBKW}, denoting 
$$ \begin{pmatrix}
       \phi_t^\eps\\
 a_t^\eps
   \end{pmatrix}=\mathcal{Z}_\eps^t\begin{pmatrix}
     \phi_0\\
 a_0
   \end{pmatrix},\quad 
   \begin{pmatrix}
      \phi^\eps(t)\\
a^\eps(t)
    \end{pmatrix}=\mathcal{S}_\eps^t\begin{pmatrix}
       \phi_0\\
 a_0
    \end{pmatrix},$$
there exist $C,c_0>0$ (depending on $\mu$) independent
  of $\eps\in (0,1]$ such that   
  \begin{equation*}
    \left\|a_t^\eps e^{i\phi_t^\eps/\eps} -a^\eps(t)
      e^{i\phi^\eps(t)/\eps} \right\|_{L^2}\le
    C\frac{t^2}{\eps}, \quad 0\le t\le c_0.
  \end{equation*}
\end{corollary}
\begin{proof}
  With the same notations as in Theorem \ref{theo:errlocBKW}, the Sobolev embedding of $\H^{\ell-4}_{\rho(t)}$ into $L^\infty$ ($\ell>d/2+5$) ensures 
 \begin{align*}
 \left\|a_t^\eps e^{i\phi_t^\eps/\eps} -a^\eps(t)
      e^{i\phi^\eps(t)/\eps} \right\|_{L^2} &\le \left\|a_t^\eps-a^\eps(t)\right\|_{L^2}+\left\|a^\eps(t)\(e^{i\phi_t^\eps/\eps} -
      e^{i\phi^\eps(t)/\eps}\)\right\|_{L^2}\\
      &\le \left\|A^\eps(t)\right\|_{L^2}+\frac{1}{\eps}\left\|a^\eps(t)\right\|_{L^\infty}\left\|\Psi^\eps(t)\right\|_{L^2}\le \frac{Ct^2}{\eps}
  \end{align*}
\end{proof}
This result will not be used in the sequel, but shows how a $1/\eps$
factor appears when going back to the wave function, in agreement with
the observations in \cite{BJM2}. The above computation also shows how
to infer the first point in Corollary~\ref{cor:wave} from
Theorem~\ref{theo:main}.

\section{Lady Windermere's fan}
\label{sec:lady}

Let $M_0>0$, $\ell>d/2+5$, and
$v_0=(\phi_0,a_0)\in\H_{M_0}^{\ell+1}\times\H_{M_0}^{\ell}$. For the
sake of conciseness, we use the following notations: for $t>0$,
$n\in\N$ and $\Delta t>0$,  
$$v^\eps(t)=(\phi^\eps(t),a^\eps(t))=\mathcal{S}_\eps^t v_0,\quad  v_n
^\eps=(\phi_n^\eps,a_n^\eps)=\(\mathcal{Z}_\eps^{\Delta t}\)^n v_0.$$  
For $\rho>0$ and $v=(\phi,a)\in
\H_{\rho}^{\ell+1}\times\H_{\rho}^{\ell}$, we also denote 
$$\|v\|_{\rho,\ell}=\|\phi\|_{\H_{\rho}^{\ell+1}}+\|a\|_{\H_{\rho}^{\ell}}.$$
According to Proposition \ref{prop:local}, if $M>0$ is sufficiently
large, $T<M_0/M$ and $\rho(t)=M_0-Mt$,
\eqref{eq:syst-grenier}-\eqref{eq:ci} has a unique solution $v^\eps\in
C([0,T],\H_\rho^{\ell+1}\times\H_\rho^{\ell})$, with 
$$\sup_{0\le t\le T}\|v^\eps(t)\|_{\rho(t),\ell}\le R,$$
where $R=2\|v_0\|_{M_0,\ell}$.

We recall the notation $t_n=n\Delta t$, and we set $\rho_n=\rho(t_n)$.
We now prove by induction on $n$ that
there exists $c_0>0$ such that if $\Delta t\in (0,c_0]$, for every
$n\ge 0$ such that $n\Delta t\le T$, we have 
\begin{align}
&\|v_n^\eps\|_{\rho_n,\ell-4}\le R+\delta,\label{eq:rec1}\\
&\|v_n^\eps-v^\eps(t_n)\|_{\rho_n,\ell-4}\le \gamma\Delta t,\label{eq:rec2}\\
&\|v_n^\eps\|_{\rho_n,\ell}\le R/2,\label{eq:rec3}
\end{align}
for some $\delta,\gamma>0$ that will be given
later. \eqref{eq:rec1}$_n$-\eqref{eq:rec2}$_n$-\eqref{eq:rec3}$_n$ obviously hold
 for $n=0$. Let $n>0$ such that $n\Delta t\le T$ and assume that
\eqref{eq:rec1}$_j$-\eqref{eq:rec2}$_j$-\eqref{eq:rec3}$_j$ hold
for all $j\in\{0,\cdots,n-1\}$. Then, for all $j\in\{0,\cdots,n-2\}$,
\eqref{eq:rec1}$_{j+1}$ yields 
\begin{equation}\label{eq:zvj}
\left\|\mathcal{Z}_\eps^{\Delta
    t}v_j^\eps\right\|_{\rho_{j+1},\ell-4}=\left\|v_{j+1}^\eps\right\|_{\rho_{j+1},\ell-4}\le
R+\delta.  
\end{equation}
On the other hand, for $j\in\{0,\cdots,n-2\}$, we also have
\begin{align}
\left\|\mathcal{S}_\eps^{\Delta t}v_j^\eps\right\|_{\rho_{j+1},\ell-4}
  & \le\left\|\mathcal{S}_\eps^{\Delta
    t}v_j^\eps-\mathcal{S}_\eps^{\Delta
    t}v^\eps(t_j)\right\|_{\rho_{j+1},\ell-4}+
\left\|v^\eps(t_{j+1})\right\|_{\rho_{j+1},\ell-4}\nonumber\\ 
&\le\left\|\mathcal{S}_\eps^{\Delta
  t}v_j^\eps-\mathcal{S}_\eps^{\Delta
  t}v^\eps(t_j)\right\|_{\rho_{j+1},\ell-4}+R.\nonumber 
\end{align}
From \eqref{eq:rec3}$_j$,
$\left\|v_j^\eps\right\|_{\rho_{j},\ell-4} \le R/2$, whereas
$\left\|v^\eps(t_j)\right\|_{\rho_{j},\ell-4} \le R$ by choice of
$R$. Thus, since $\ell-4>d/2+1$, Proposition \ref{prop:local} and
\eqref{eq:rec2}$_j$ imply (up to increasing $M$) 
\begin{equation*}
\left\|\mathcal{S}_\eps^{\Delta t}v_j^\eps-\mathcal{S}_\eps^{\Delta
    t}v^\eps(t_j)\right\|_{\rho_{j+1},\ell-4}\le K(R)\gamma\Delta t. 
\end{equation*}
Therefore, if $c_0>0$ is chosen sufficiently small such that
$K(R)\gamma c_0\le \delta$, we have  
\begin{equation}\label{eq:svj}
\left\|\mathcal{S}_\eps^{\Delta t}v_j^\eps\right\|_{\rho_{j+1},\ell-4}\le R+\delta,
\end{equation}
and \eqref{eq:zvj}, \eqref{eq:svj} and Proposition \ref{prop:local}
ensure that  for all $j\in\{0,\cdots,n-2\}$, 
\begin{equation}
\left\|\(\mathcal{S}_\eps^{\Delta t}\)^{n-1-j}\mathcal{Z}_\eps^{\Delta
    t}v_j^\eps-\(\mathcal{S}_\eps^{\Delta
    t}\)^{n-1-j}\mathcal{S}_\eps^{\Delta
    t}v_j^\eps\right\|_{\rho_{n},\ell-4}\le
K(R+\delta)\left\|\mathcal{Z}_\eps^{\Delta
    t}v_j^\eps-\mathcal{S}_\eps^{\Delta
    t}v_j^\eps\right\|_{\rho_{j+1},\ell-4}.\nonumber 
\end{equation}
Moreover, the last estimate also holds for $j=n-1$ if $K$ is replaced
by 1. According to \eqref{eq:rec3}$_j$ and Theorem
\ref{theo:errlocBKW}, we deduce that for all $j\in\{0,\cdots,n-1\}$, 
\begin{equation}
\left\|\(\mathcal{S}_\eps^{\Delta t}\)^{n-1-j}\mathcal{Z}_\eps^{\Delta
    t}v_j^\eps-\(\mathcal{S}_\eps^{\Delta
    t}\)^{n-1-j}\mathcal{S}_\eps^{\Delta
    t}v_j^\eps\right\|_{\rho_{n},\ell-4}\le
\max(1,K(R+\delta))C(R/2)\Delta t^2. 
\end{equation}
Piling up the last inequality for $j\in\{0,\cdots,n-1\}$, we conclude
\begin{align*}
\|v_n^\eps-v^\eps(t_n)\|_{\rho_n,\ell-4}
  &\le\sum_{j=0}^{n-1}\left\|\(\mathcal{S}_\eps^{\Delta
    t}\)^{n-1-j}\mathcal{Z}_\eps^{\Delta
    t}v_j^\eps-\(\mathcal{S}_\eps^{\Delta
    t}\)^{n-1-j}\mathcal{S}_\eps^{\Delta
    t}v_j^\eps\right\|_{\rho_{n},\ell-4}\\ 
 &\le n\max(1,K(R+\delta))C(R/2)\Delta t^2\\
 &\le\max(1,K(R+\delta))C(R/2)T\Delta t,
\end{align*}
which proves \eqref{eq:rec2}$_n$ with
$\gamma=\max(1,K(R+\delta))C(R/2)T$. Then, \eqref{eq:rec2}$_n$ yields 
\begin{equation}\label{eq:presque-rec1}
\|v_n^\eps\|_{\rho_n,\ell-4} \le
\|v_n^\eps-v^\eps(t_n)\|_{\rho_n,\ell-4}
+\|v^\eps(t_n)\|_{\rho_n,\ell-4} \le\gamma\Delta t+R. 
\end{equation}
Note that it does not prove \eqref{eq:rec1}$_n$ yet, because the
choice of $\delta=\gamma c_0$ may be incompatible with the previous
constraint $K(R)\gamma c_0\le \delta$. However, \eqref{eq:rec3}$_n$
follows from \eqref{eq:presque-rec1} and Corollary \ref{cor:borneZ},
once we have noticed that the proof of \eqref{eq:presque-rec1} also
works if $v_n^\eps=\mathcal{Z}_\eps^{\Delta t}v_{n-1}^\eps$ is
replaced by $\mathcal{Z}_\eps^{t}v_{n-1}^\eps$ (and $t_n$ by
$t_{n-1}+t$), for any $0\le t\le\Delta t$, so that 
\begin{align*}
&\mathcal{Z}_\eps^{t}\(\mathcal{Z}_\eps^{\Delta t}\)^{n-1} v_0 -\mathcal{S}_\eps^{t+(n-1)\Delta t} v_{0}\\
&\quad= \mathcal{Z}_\eps^{t}v_{n-1}^\eps -\mathcal{S}_\eps^{t}v_{n-1}^\eps
+\sum_{j=0}^{n-2}\left[\mathcal{S}_\eps^{t} \(\mathcal{S}_\eps^{\Delta t}\)^{n-2-j}\mathcal{Z}_\eps^{\Delta t}v_j^\eps-\mathcal{S}_\eps^{t}\(\mathcal{S}_\eps^{\Delta t}\)^{n-2-j}\mathcal{S}_\eps^{\Delta t}v_j^\eps\right].
\end{align*}
Then, \eqref{eq:rec1}$_n$ follows from \eqref{eq:rec3}$_n$, and any
positive value for $\delta$ is admissible.

\bibliographystyle{siam}
\bibliography{splitting}

\end{document}